\documentclass[preprint,12pt]{elsarticle}

\usepackage{graphicx}
\usepackage{amssymb}
\usepackage{amsthm}
\newtheorem{theorem}{Theorem}
\newtheorem{proposition}{Proposition}
\newtheorem{corollary}{Corollary}
 
\usepackage{lineno}

\usepackage{amsfonts,amsmath,graphicx,subfigure,hyperref,cleveref,bm}
\usepackage{algorithm,algpseudocode}
\usepackage{booktabs,comment,multirow,todonotes}
\usepackage[flushleft]{threeparttable} 

\usepackage{color}

\journal{Journal of Computational Physics}

\begin{document}

\begin{frontmatter}


\title{A Robust Hierarchical Solver for Ill-conditioned Systems with Applications to Ice Sheet Modeling}



\author[label1]{Chao\ Chen}
\ead{chenchao.nk@gmail.edu}

\author[label1]{Leopold\ Cambier}
\ead{lcambier@stanford.edu}

\author[label2]{Erik~G. Boman}
\ead{egboman@sandia.gov}

\author[label2]{Sivasankaran\ Rajamanickam}
\ead{srajama@sandia.gov}

\author[label2]{Raymond~S. Tuminaro}
\ead{rstumin@sandia.gov}

\author[label1]{Eric\ Darve}
\ead{darve@stanford.edu}

\address[label1]{Stanford University}
\address[label2]{Sandia National Laboratories}

\begin{abstract}
A hierarchical solver is proposed for solving sparse ill-conditioned linear systems in parallel. The solver is based on a modification of the LoRaSp method, but employs a deferred-compression technique, which provably reduces the approximation error  and significantly improves efficiency. Moreover, the deferred-compression technique introduces minimal overhead and does not affect  parallelism. As a result, the new solver achieves linear computational complexity under mild assumptions and excellent parallel scalability.
To demonstrate the performance of the new solver, we focus on applying it to solve sparse linear systems arising from ice sheet modeling. The strong anisotropic phenomena associated with the thin structure of ice sheets creates serious challenges for existing solvers.
To address the anisotropy, we additionally developed a customized partitioning scheme for the solver, which captures the strong-coupling direction accurately. In general, the partitioning can be computed algebraically with existing software packages, and thus the new solver is generalizable for solving other sparse linear systems. Our results show that ice sheet problems of about 300 million degrees of freedom have been solved in just a few minutes using a thousand processors.
\end{abstract}
 

\end{frontmatter}


\section{Introduction} \label{CC:sec:intro}

This paper considers the problem of solving large sparse linear systems, which is a fundamental building block but also often a computational bottleneck in many science and engineering applications. In particular, we target linear systems that result from the numerical discretization of elliptic partial differentials equations (PDE) including Laplace, Stokes, Helmholtz equations (in low and middle frequency regime), etc., using local schemes such as finite differences or finite elements. One challenge arises when the condition number of the problem is large, and existing solvers become inefficient. Existing solvers fall into three classes. The first class is sparse direct solvers~\cite{CC:davis2016survey}, which leverage efficient ordering schemes to perform Gaussian elimination. However, they generally require $\mathcal{O}(N^2)$ computation and $\mathcal{O}(N^{4/3})$ storage for solving a three-dimensional problem of size $N$. These large costs seriously limit the application of sparse direct solvers to truly large-scale problems. The second class is iterative solvers such as the conjugate gradient method and the multigrid method. These methods may need only $\mathcal{O}(N)$ work and storage per iteration. However, the number of iterations required to achieve convergence can be quite large when solving ill-conditioned linear systems. Preconditioning is essential to improve the conditioning and convergence.
 
The third class of methods, which is the focus here, is hierarchical solvers. These methods compute an approximate factorization of a discretized elliptic PDE by taking advantage of the fact that the discretization matrix (and its inverse) has certainty hierarchical low-rank structures, including $\mathcal{H}$-~\cite{CC:hackbusch1999sparse, CC:hackbusch2000sparse}, $\mathcal{H}^2$-~\cite{CC:hackbusch2002data, CC:hackbusch2015mathcal} matrices and hierarchically semiseparable (HSS)~\cite{CC:xia2010fast, CC:chandrasekaran2006fast} matrices, among others~\cite{CC:amestoy2015improving, CC:aminfar2016fast}. By exploiting this data sparsity of the underlying physical problem, hierarchical solvers have been shown to achieve linear or quasi-linear complexity. However, their efficiency deteriorates when highly ill-conditioned problems are encountered because the rank/costs of approximations must increase dramatically in order to maintain the same accuracy in the final solution. Note that hierarchical solvers can be used either as direct solver (high accuracy) or as a preconditioner (low accuracy) for iterative methods. Our focus is on the latter.


In this paper, we introduce the deferred-compression technique to improve the efficiency of hierarchical solvers for solving sparse ill-conditioned linear systems and demonstrate this improvement by implementing it in a particular hierarchical solver named LoRaSp~\cite{CC:pouransari2017fast}. In hierarchical solvers such as LoRaSp, off-diagonal matrix blocks are compressed with low-rank approximations if they satisfy the strong admissibility condition~\cite{CC:hackbusch2002data}. In our new solver, these compressible matrix blocks are first scaled by Cholesky factors of the corresponding diagonal blocks before low-rank approximations are applied. This extra scaling step provably reduces errors in the subsequent step of forming the Schur complement.
In addition, it increases the likelihood that the Schur complement
remains symmetric positive definite with crude low-rank approximation when the original input matrix is SPD. For many practical applications, using the deferred-compression technique to preserve the SPD property of the underlying physical problem is crucial. 

Previous deferred-compression work~\cite{CC:xia2010robust,CC:xia2017effective,CC:xing2018preserving} focused on HSS matrices, which is a type of weakly admissible (as opposed to strongly admissible) hierarchical matrices. These approaches are not directly applicable to strongly admissible hierarchical matrices (e.g., $\mathcal{H}^2$ matrices) such as the LoRaSp solver. Furthermore, prior deferred-compression efforts concentrated on solving dense linear systems, where incorporating the deferred-compression step leads to an extra $\mathcal{O}(N^2)$ amount of computation, and no corresponding parallel solver was developed. Compared to the previously published papers, our deferred-compression technique is novel in three ways:
\begin{enumerate}
    \item we target hierarchical solvers specialized for strongly admissible hierarchical matrices, and develop an associated general error analysis; the previous analysis can be recovered as a special case of our new analysis.
    \item we propose a new solver for sparse linear systems for which it is proved that the computational complexity is $\mathcal{O}(N)$ under some mild assumptions. This nearly optimal complexity implies that we can solve large problems with minimum asymptotic computational cost (up to some constants).
    \item  we show that incorporating the deferred-compression scheme into the LoRaSp solver does not change the data and task dependencies in the parallel LoRaSp solver~\cite{CC:chen2016parallel}. Therefore, we can take advantage of the existing parallel algorithm to solve large-scale problems efficiently (on distributed-memory machines).
\end{enumerate}

In order to demonstrate the performance of our new solver, this paper addresses the challenges of solving linear systems from a real-world problem---ice sheet modeling where the solution of discretized linear systems remains the computational bottleneck. Ice sheet modeling is an essential component needed to estimate future sea-level rise due to climate modeling. As noted in~\cite{CC:solomon2007climate, CC:stocker2014climate} from the Intergovernmental Panel on Climate Change (IPCC), modern ice sheet models must continue to introduce advanced features such as adaptive mesh refinement at sub-kilometer resolutions, optimization, data assimilation and uncertainty quantification for the treatment of numerous unknown model inputs. These advances will likely introduce further computational burdens requiring improvements in the linear solver, which must be repeatedly invoked over the course of the simulation.  Given that current ice sheet simulations already consume resources on thousands of processing units on modern supercomputers and can involve up to billions of unknown variables, there is a pressing need for efficient linear solvers to reduce simulation costs and prepare for potentially larger more sophisticated simulations in the future.

However, many existing solvers turn out to deliver rather disappointing performance for solving problems from ice sheet modeling. The most prominent challenge comes from the anisotropic nature of ice sheet models, where the thin vertical scale of the domain is tiny relative to the horizontal scale. This extraordinary contrast is also reflected by the dramatically different magnitudes of entries in the discretization matrix, where large values correspond to strong vertical coupling and tiny ones to weak horizontal coupling. This weak coupling gives rise to oscillatory eigenvectors associated with small eigenvalues and a poorly-conditioned linear system. This can be seen from a simplified model $\epsilon u_{xx} + u_{yy}$, where $\epsilon \ll 1$, where the standard five point finite difference discretization on a $n \times n$ regular grid produces a matrix with many small eigenvalues \[4(n+1)^2 [\epsilon \sin^2(\pi i/(2n+2)) + \sin^2(\pi j/(2n+2))]\] for \textit{all} values of $i$ and small values of $j$. Further, the Neumann boundary condition imposed on the top surface and some bottom parts of the domain gives rise to problematic linear systems with nearly singular matrix blocks. Physically, the bottom Neumann boundary condition models large ice shelves, which are vast areas of floating ice connected to land-based ice sheets and are common to Antarctica. The resulting Green's function decays much slower along the vertical direction than that for non-sliding ice at a frozen ice interface~\cite{CC:tuminaro2016matrix}, again contributing to the poor performance of many existing solvers.



The two solvers (preconditioners) commonly used in ice sheet modeling are the incomplete LU factorization (ILU) and the algebraic multigrid method (AMG). Although the customized ILU with a specific ordering scheme performs reasonably well for the Greenland ice sheet problem, its performance deteriorates significantly for the Antarctic ice sheet problem. The reason is that ice sheets on the Antarctic problem contain a substantial fraction of floating ice shelves, modeled by imposing Neumann boundary conditions, which leads to aforementioned ill-conditioned linear systems. Another possible approach to solve the ice sheet linear systems is some form of algebraic multigrid (AMG). However, standard AMG methods (e.g., the smoothed aggregation AMG solver~\cite{CC:vanvek1996algebraic}) do not generally converge on realistic ice sheet simulations. While some specialized AMG techniques have been successfully developed (e.g., a customized matrix-dependent AMG solver~\cite{CC:tuminaro2016matrix}) using tailored semi-coarsening schemes, these approaches required significant non-trivial multigrid adaptions to address ice sheet simulations. These types of adaptations are not generally provided with available AMG packages.


To solve the particular linear systems from ice sheet modeling efficiently, our new solver introduces one customization to efficiently address the ice sheet linear systems. Specifically, the typical meshes employed for ice sheet models are generated by first creating a two-dimensional unstructured horizontal mesh and then extruding this mesh into the vertical dimension to create a three-dimensional grid. This mesh structure is leveraged when building clusters for the hierarchical solver. 
In particular, the (two-dimensional unstructured) non-extruded mesh is first partitioned with a general graph partitioner and then the horizontal partition results are extended along the third/extruded direction such that mesh vertices lying on the same vertical line belong to the same cluster. Since extruded meshes appear frequently in geophysical problems such as atmospheric and oceanic circulation, oil and gas modeling, etc., our new solver along with the ``extruded partitioning'' algorithm can be generally applied to other engineering simulations involving thin structures. 
%
%
Compared with the ILU and the AMG methods used for ice sheet modeling, our new solver is robust in the sense that the iteration number stays nearly constant if it is used as a preconditioner for solving linear systems associated with ice sheet modeling, and our new solver is general-purpose in that it can be applied as a ``black-box'' method with a general partitioning scheme available in several existing software packages, such as METIS/ParMETIS~\cite{CC:karypis1998fast}, Scotch~\cite{CC:chevalier2008pt} and Zoltan~\cite{CC:boman2012zoltan}, though a special partitioner can also be easily incorporated. Moreover, it is challenging to parallelize the ILU and the AMG methods on modern many-core architectures such as the GPU. Our new solver, similar to other hierarchical solvers, is mainly based on dense linear algebra subroutines and thus can potentially be accelerated using many-core processors.

To summarize, the paper presents a parallel hierarchical solver for sparse ill-conditioned linear systems using the deferred-compression technique, and in particular, our work makes the following three major contributions:
\begin{enumerate}
\item Error analysis of the deferred-compression scheme for hierarchical solvers based on strongly admissible hierarchical matrices (e.g., $\mathcal{H}^2$-matrices).
\item A parallel/distributed-memory hierarchical solver for sparse ill-conditioned linear systems.
\item Application and analysis of the preconditioner for an ice sheet modeling problem, including numerical comparisons with ILU.\footnote{A high-performance implementation in the Trilinos IFPACK package.}
\end{enumerate}



The remainder of this paper is organized as follows. Section~\ref{CC:sec:diagonal_scaling} introduces the deferred-compression technique and provides an error analysis. Following that is the algorithm of our new solver presented in Section~\ref{CC:sec:algorithm}. Next Section~\ref{CC:sec:stokes} briefly summarizes the first-order-accurate Stokes approximation model of ice sheets and introduces the ``extruded partitioning'' algorithm. Finally, in Section~\ref{CC:sec:results} numerical results are given demonstrating the performance and scalability of our new solver for ice sheet modeling and also general problems from the SuiteSparse Matrix Collection.\footnote{https://sparse.tamu.edu/}

\section{Deferred-compression Scheme} \label{CC:sec:diagonal_scaling}

This section presents the algorithm for deferred-compression and the corresponding error analysis, targeted at hierarchical solvers that are based on strongly admissible hierarchical matrices. These solvers employ low-rank approximations to compress off-diagonal matrix blocks that satisfy the strong-admissibility condition. A rigorous definition of the strong-admissibility condition can be found in~\cite{CC:hackbusch2002data}. From a high-level perspective, the strong-admissibility condition states that one block-row in a (appropriately partitioned) strongly admissible hierarchical matrix includes a diagonal block corresponding to ``self-interaction,'' a full-rank off-diagonal block corresponding to ``neighbor or near-field interaction,'' and a (numerically) low-rank off-diagonal block corresponding to ``well-separated or far-field interaction.''
Therefore, a strongly admissible hierarchical matrix $A$ can be partitioned as the following $3 \times 3$ block matrix
\[
A = 
\begin{pmatrix}
A_{ss} & A_{sn} & A_{sw} \\
A_{ns} & A_{nn} & A_{nw} \\
A_{ws} & A_{wn} & A_{ww}
\end{pmatrix}
\]
where ``s'' is a set of row/column indexes that we seek to eliminate via a Cholesky factorization. ``n'' stands for the set of indexes for which $A_{ns}$ and $A_{sn}$ are full rank, and ``w'' is used to denote the low-rank blocks $A_{ws}$ and $A_{sw}$. We further assume that $A$ is SPD in this paper, so $A_{ns} = A_{sn}^T$, $A_{ws} = A_{sw}^T$, and $A_{wn} = A_{nw}^T$.



Below we first review the classical Cholesky factorization and introduce some notations; then we analyze the errors in forming (approximate) Schur complements when the $s$ block is eliminated with and without using the deferred-compression scheme. In order to measure error, we use the matrix-norm (a.k.a., 2-norm or operator norm) denoted by $\|\cdot\|$.

\paragraph{Cholesky factorization} To carry out one step of (block) Cholesky factorization on the $s$ block in $A$, we define the following three matrices 
\begin{align*}
{\cal S}_s = 
\begin{pmatrix}
G_s^{-1} & &  \\
& I & \\
& & I
\end{pmatrix}
\qquad
{\cal L}_s = 
\begin{pmatrix}
I & & \\
- A_{ns} G_s^{-T} & I & \\
- A_{ws} G_s^{-T} & & I
\end{pmatrix}
\qquad 
{\cal C}_s = {\cal L}_s {\cal S}_s
\end{align*}
where $A_{ss} = G_s G_s^T$ is the Cholesky factorization of $A_{ss}$. The (exact) Schur complement $S_{A}$ is found in the lower $2 \times 2$ block matrix of ${\cal C}_s A {\cal C}_s^T$ as
\begin{equation} \label{eqn:exact_sc}
S_A = 
\begin{pmatrix}
A_{nn} - A_{ns} A_{ss}^{-1} A_{sn} & A_{nw} - A_{ns} A_{ss}^{-1} A_{sw} \\
A_{wn} - A_{ws} A_{ss}^{-1} A_{sn} & A_{ww} - A_{ws} A_{ss}^{-1} A_{sw}
\end{pmatrix}
\end{equation}

To actually compute the Cholesky factorization of the whole matrix $A$, the Schur complement $S_A$ needs to be further factorized, which is skipped here since this is not relevant for our current discussion.

\paragraph{\textbf{Without} deferred-compression scheme} 
Suppose the low-rank matrix block $A_{sw}$ can be decomposed as 
\begin{equation}\label{eqn:offdiag}
A_{sw} = U V^T = 
\begin{pmatrix}
U_1 & U_2
\end{pmatrix}
\begin{pmatrix}
V_1^T \\
V_2^T 
\end{pmatrix}
= U_1 V_1^T + U_2 V_2^T, 
\end{equation}
where $U$ is an orthogonal matrix and $\|V_2^T\|_2 = \epsilon$, a small prescribed tolerance. This kind of decomposition can be computed using, e.g., a rank-revealing QR factorization (RRQR). Dropping the $U_2 V_2^T$ term in $A$ leads to the compressed matrix $\tilde{A} = compress(A)$ as follows
\[
A \approx \tilde{A} = 
\begin{pmatrix}
A_{ss} & A_{sn} & U_1 V_1^T \\
A_{ns} & A_{nn} & A_{nw} \\
V_1 U_1^T & A_{wn} & A_{ww}
\end{pmatrix},
\]
where the low-rank approximation can be exploited to compute an approximate factorization of $A$ at a lower cost. Apply one step of the Cholesky factorization on the $s$ block in $\tilde{A}$ with 
\[
\tilde{{\cal C}}_s = \tilde{{\cal L}}_s {\cal S}_s
\]
where
\[
\tilde{{\cal L}}_s =
\begin{pmatrix}
I & & \\
- A_{ns} G_s^{-T} & I & \\
- V_1 U_1^T G_s^{-T} & & I
\end{pmatrix}.
\]
As a result, $\tilde{{\cal C}}_s \tilde{A} \tilde{{\cal C}}_s^T$ contains $S_{\tilde{A}}$, an approximation of $S_A$ with error $\tilde{E}$:
\begin{align} \label{eqn:sc_old}
S_{\tilde{A}} & = 
\begin{pmatrix}
A_{nn} - A_{ns} A_{ss}^{-1} A_{sn} & A_{nw} - A_{ns} A_{ss}^{-1} (U_1 V_1^T) \\
A_{wn} - (V_1 U_1^T) A_{ss}^{-1} A_{sn} & A_{ww} - (V_1 U_1^T) A_{ss}^{-1} (U_1 V_1^T)
\end{pmatrix}, 
\end{align}
\begin{align} \label{eqn:err_old}
\tilde{E} = &
S_{\tilde{A}} - S_A = 
\begin{pmatrix}
0 & A_{ns} A_{ss}^{-1} (U_2 V_2^T) \\
(V_2 U_2^T) A_{ss}^{-1} A_{sn} & A_{ws} A_{ss}^{-1} A_{sw} - (V_1 U_1^T) A_{ss}^{-1} (U_1 V_1^T)
\end{pmatrix}.
\end{align}

\begin{proposition}
Assume Eq.(\ref{eqn:offdiag}) holds, the error $\tilde{E}$ between the two Schur complements, namely $S_{\tilde{A}}$ in $\tilde{{\cal C}}_s \tilde{A} \tilde{{\cal C}}_s^T$ and $S_{A}$ in ${\cal C}_s A {\cal C}_s^T$ takes the form in Eq.(\ref{eqn:err_old}). Moreover, the following error estimates hold \vspace{2mm}
\begin{enumerate}
\setlength\itemsep{0.6em}
\item $\|\tilde{E}_{ww}\| \le 2 \epsilon \|A_{sw}\| / \sigma_{\min}(A_{ss}) + O(\epsilon^2)$,
\item $\|\tilde{E}_{nw}\| = \|\tilde{E}_{wn}\| \le \epsilon \|A_{ns}\| / \sigma_{\min}(A_{ss})$, 
\item and $\|\tilde{E}\| \le \|\tilde{E}_{ww}\| + \|\tilde{E}_{nw}\| \le \epsilon (2 \|A_{sw}\| + \|A_{ns}\|) / \sigma_{\min}(A_{ss}) + O(\epsilon^2)$,
\end{enumerate}
\vspace{2mm}
where $\tilde{E}_{nw}, \tilde{E}_{ww}$ and $\tilde{E}_{ww}$ stand for the (1,2) block, (2,1) block and (2,2) block in $\tilde{E}$.
\end{proposition}
\vspace{0.1mm} 
\begin{proof}
The first part of the proposition is already shown above, so we derive the three error bounds as follows.
\begin{align} \label{eqn:old_error}
1. \;
 \|\tilde{E}_{ww}\|
=& \| A_{ws} A_{ss}^{-1} A_{sw} - (V_1 U_1^T) A_{ss}^{-1} (U_1 V_1^T) \| \nonumber \\[0.7ex] 
=& \| (V_1 U_1^T + V_2 U_2^T) A_{ss}^{-1} (U_1 V_1^T + U_2 V_2^T) - (V_1 U_1^T) A_{ss}^{-1} (U_1 V_1^T) \| \nonumber \\[0.5ex]
=& \| V_1 U_1^T (A_{ss})^{-1} U_2 V_2^T + V_2 U_2^T (A_{ss})^{-1} (U_1 V_1^T + U_2 V_2^T) \| \nonumber \\[0.7ex] 
\le& \| V_1 U_1^T (A_{ss})^{-1} U_2 V_2^T \| + \| V_2 U_2^T (A_{ss})^{-1} (U_1 V_1^T + U_2 V_2^T) \| \nonumber \\[0.7ex]
\le& \| V_1 U_1^T \| \|(A_{ss})^{-1}\| \|U_2 V_2^T \| + \| V_2 U_2^T\| \|(A_{ss})^{-1}\| \|(U_1 V_1^T + U_2 V_2^T) \| \nonumber \\[0.7ex]
=& \epsilon (\|A_{sw}\| + \epsilon) / \sigma_{\min}(A_{ss}) + \epsilon \|A_{sw}\| / \sigma_{\min}(A_{ss}) \nonumber \\[0.7ex]
=& 2 \epsilon \|A_{sw}\| / \sigma_{\min}(A_{ss}) + \epsilon^2 / \sigma_{\min}(A_{ss})
\end{align}
where Eq.~(\ref{eqn:offdiag}), $ \|U_2V_2^T\| = \epsilon$ and $\|U_1V_1^T\| = \|A_{sw} - U_2V_2^T\| \le \|A_{sw} \| + \epsilon $ are used, and $\sigma_{\min}$ denotes the smallest singular value of a matrix.

\begin{align*}
2. \quad  \|\tilde{E}_{nw}\| = \|\tilde{E}_{wn}\| = \|A_{ns} A_{ss}^{-1} (U_2 V_2^T)\| \le \epsilon \|A_{ns}\| / \sigma_{\min}(A_{ss}) \qquad \qquad \qquad 
\end{align*}

\begin{align*}
3. \quad \|\tilde{E}\| &= 
\left\vert\left\vert
\begin{pmatrix}
0 & \tilde{E}_{nw} \\
\tilde{E}_{wn} & \tilde{E}_{ww}
\end{pmatrix}
\right\vert\right\vert \\[0.8ex]
&= \left\vert\left\vert
\begin{pmatrix}
0 & \tilde{E}_{nw} \\
\tilde{E}_{wn} & 0
\end{pmatrix}
+
\begin{pmatrix}
0 & 0 \\
0 & \tilde{E}_{ww}
\end{pmatrix}
\right\vert\right\vert \\[0.8ex]
&\le \left\vert\left\vert
\begin{pmatrix}
0 & \tilde{E}_{nw} \\
\tilde{E}_{wn} & 0
\end{pmatrix}
\right\vert\right\vert
+
\left\vert\left\vert
\begin{pmatrix}
0 & 0 \\
0 & \tilde{E}_{ww}
\end{pmatrix}
\right\vert\right\vert \\[0.8ex]
&= \|\tilde{E}_{nw}\| + \|\tilde{E}_{ww}\| \\[0.8ex]
&\le \epsilon (2 \|A_{sw}\| + \|A_{ns}\|) / \sigma_{\min}(A_{ss}) + \epsilon^2 / \sigma_{\min}(A_{ss}) 
\qquad \qquad \qquad \qquad 
\end{align*}
where we used the equality 
$
\left\vert\left\vert
\begin{pmatrix}
0 & C \\
C^T & 0
\end{pmatrix}
\right\vert\right\vert = \|C\|
$ for any matrix $C$.
\hfill \end{proof}

For ill-conditioned problems such as linear systems from ice sheet modeling the diagonal matrix block $A_{ss}$ can be nearly singular, and so $\sigma_{\min}(A_{ss})$ is very small. As a result, the error $\|\tilde{E}\|$ can be large. Worse still, due to this large error the approximate Schur complement $S_{\tilde{A}}$ may become indefinite and the Cholesky factorization of diagonal blocks can break down. This leads to a poor approximation of the exact Schur Complement $S_{A}$, an SPD matrix.

The above error analysis extends to all hierarchical solvers based on strongly admissible hierarchical matrices (with potentially minor modifications) and shows that the low-rank truncation error $\epsilon$ needs to decrease at least as fast as $\sigma_{\min}(A_{ss})$ to maintain the same error tolerance on $S_{\tilde{A}}$.

\paragraph{\textbf{With} deferred-compression scheme} Before compressing the off-diagonal matrix block $A_{sw}$ directly, we first scale $A_{sw}$ by the inverse of the Cholesky factor of $A_{ss}$. Specifically, the Cholesky factorization of the diagonal block $A_{ss} = G_s G_s^T$ is used to scale the first block row and column of $A$ as the following
\[
{\cal S}_s A {\cal S}_s^T
= 
\begin{pmatrix}
I & G_s^{-1} A_{sn} & G_s^{-1} A_{sw} \\
A_{ns} G_s^{-T} & A_{nn} & A_{nw} \\
A_{ws} G_s^{-T} & A_{wn} & A_{ww}
\end{pmatrix},
\]
where 
${\cal S}_s = 
\begin{pmatrix}
G_s^{-1} & &  \\
& I & \\
& & I
\end{pmatrix}
$.
Then the block $G_s^{-1} A_{sw}$ is compressed with a low-rank approximation. Similar to Eq.~(\ref{eqn:offdiag}), assume a low-rank decomposition of $G^{-1} A_{sw}$ as
\begin{align}\label{eqn:offdiag_new}
G_s^{-1} A_{sw} = \hat{U} \hat{V}^T = 
\begin{pmatrix}
\hat{U}_1 & \hat{U}_2
\end{pmatrix}
\begin{pmatrix}
\hat{V}_1^T \\
\hat{V}_2^T 
\end{pmatrix}
= \hat{U}_1 \hat{V}_1^T + \hat{U}_2 \hat{V}_2^T, 
\end{align}
where $\hat{U}$ is orthogonal and $\|\hat{V}_2\| = \epsilon$, a small prescribed tolerance. 
One way to relate Eq.~(\ref{eqn:offdiag_new}) to Eq.~(\ref{eqn:offdiag}) is the following. Define $\bar{U} = G_s \hat{U}$, then Eq.~(\ref{eqn:offdiag_new}) is equivalent to $A_{sw} = \bar{U} \hat{V}^T$, where $\bar{U}$ is orthogonal in terms of the inner product defined by the SPD matrix $A_{ss}^{-1}$.

Replacing $G_s^{-1} A_{sw}$ by $\hat{U}_1 \hat{V}_1^T$ in ${\cal S}_s A {\cal S}_s^T$ leads to the compressed matrix ${\cal S}_s \hat{A} {\cal S}_s^T = compress({\cal S}_s A {\cal S}_s^T)$ as follows
\[
{\cal S}_s A {\cal S}_s^T \approx {\cal S}_s \hat{A} {\cal S}_s^T = 
\begin{pmatrix}
I & G_s^{-1} A_{sn} & \hat{U}_1 \hat{V}_1^T \\
A_{ns} G_s^{-T} & A_{nn} & A_{nw} \\
\hat{V}_1 \hat{U}_1^T & A_{wn} & A_{ww}
\end{pmatrix},
\]
where
\[
\hat{A} = 
\begin{pmatrix}
A_{ss} & A_{sn} & G_s \hat{U}_1 \hat{V}_1^T \\
A_{ns} & A_{nn} & A_{nw} \\
\hat{V}_1 \hat{U}_1^T G_s^T & A_{wn} & A_{ww}
\end{pmatrix}.
\]
Carrying out one step of Cholesky factorization on the $s$ block in ${\cal S}_s \hat{A} {\cal S}_s^T$ with
\[
\hat{{\cal L}}_s =
\begin{pmatrix}
I & & \\
- A_{ns} G_s^{-T} & I & \\
- \hat{V}_1 \hat{U}_1^T & & I
\end{pmatrix}
\]
produces the following Schur complement, another approximation of $S_A$ as follows
\begin{align} \label{eqn:sc_new}
S_{{\cal S}_s \hat{A} {\cal S}_s^T} & = 
\begin{pmatrix}
A_{nn} - A_{ns} A_{ss}^{-1} A_{sn} & A_{nw} - A_{ns} G_s^{-T} (\hat{U}_1 \hat{V}_1^T) \\
A_{wn} - (\hat{V}_1 \hat{U}_1^T) G_s^{-1} A_{sn} & A_{ww} - \hat{V}_1 \hat{V}_1^T
\end{pmatrix}.
\end{align}
\begin{proposition}
Assume Eq.(\ref{eqn:offdiag_new}) holds, the error $\hat{E}$ between the two Schur complements, namely $S_{{\cal S}_s \hat{A} {\cal S}_s^T} $ in $\hat{{\cal L}}_s {\cal S}_s \hat{A} {\cal S}_s^T \hat{{\cal L}}_s^T$ and $S_{A}$ in ${\cal C}_s A {\cal C}_s^T = {\cal L}_s {\cal S}_s A {\cal S}_s^T {\cal L}_s^T$ is the following 
\begin{align} \label{eqn:err_new}
\hat{E} & =  S_{{\cal S}_s \hat{A} {\cal S}_s^T} - S_A =
\begin{pmatrix}
0 & A_{ns} G_s^{-T} \hat{U}_2 \hat{V}_2^T \\
\hat{V}_2 \hat{U}_2^T G_s^{-1} A_{sn} &  \hat{V}_2 \hat{V}_2^T
\end{pmatrix}.
\end{align}
Moreover, the following error estimates hold \vspace{2mm}
\begin{enumerate}
\setlength\itemsep{0.6em}
\item $\|\hat{E}_{ww}\| = \epsilon^2$,
\item $\|\hat{E}_{nw}\| = \|\hat{E}_{wn}\| \le \epsilon \|A_{ns}\| / \sigma_{\min}(A_{ss})^{1/2}$, 
\item and $\|\hat{E}\| \le \|\hat{E}_{ww}\| + \|\hat{E}_{nw}\| \le \epsilon \|A_{ns}\| / \sigma_{\min}(A_{ss})^{1/2} + \epsilon^2$,
\end{enumerate}
\vspace{2mm}
where $\hat{E}_{nw}, \hat{E}_{ww}$ and $\hat{E}_{ww}$ stand for the (1,2) block, (2,1) block and (2,2) block in $\hat{E}$.
\end{proposition}
\vspace{0.1mm} 
\begin{proof}
We first show Eq.~(\ref{eqn:err_new}) as follows:
\begin{align*}
\hat{E}_{nw} &= A_{ns} (A_{ss}^{-1} A_{sw} - G_s^{-T} (\hat{U}_1 \hat{V}_1^T))  \\
&= A_{ns} G_s^{-T} (G_s^{-1} A_{sw} -  (\hat{U}_1 \hat{V}_1^T))  \\
&= A_{ns} G_s^{-T} \hat{U}_2 \hat{V}_2^T  \\ \\
\hat{E}_{ww} &= A_{ws} A_{ss}^{-1} A_{sw} - \hat{V}_1 \hat{V}_1^T  \\ 
&= (A_{ws} G_s^{-T}) (G_s^{-1} A_{sw}) - \hat{V}_1 \hat{V}_1^T  \\
&= (\hat{V}_1 \hat{U}_1^T + \hat{V}_2 \hat{U}_2^T) (\hat{U}_1 \hat{V}_1^T + \hat{U}_2 \hat{V}_2^T) - \hat{V}_1 \hat{V}_1^T  \\
&= \hat{V}_2 \hat{V}_2^T
\end{align*}
where Eq.~(\ref{eqn:offdiag_new}) and the orthogonality of $\hat{U}$ ($\hat{U}_1^T \hat{U}_1 = I$ and $\hat{U}_1^T \hat{U}_2 = 0$) are used.  Next, we can prove the following three error bounds easily.
\begin{align*}
\|\hat{E}_{ww}\| & = \|\hat{V}_2 \hat{V}_2^T\| = \epsilon^2  \\[10pt]
\|\hat{E}_{ww}\| & = \|A_{ns} G_s^{-T} \hat{U}_2 \hat{V}_2^T\| \le \epsilon \|A_{ns}\| / \sigma_{\min}(A_{ss})^{1/2} \\[10pt]
\|\hat{E}\| &\le \|\hat{E}_{ww}\| + \|\hat{E}_{nw}\| \le \epsilon \|A_{ns}\| / \sigma_{\min}(A_{ss})^{1/2} + \epsilon^2
\end{align*}
which finishes the proof.
\hfill \end{proof}

As the above proposition shows, the approximate Schur complement computed with the deferred-compression scheme is much more accurate than that without the scheme, especially when the problem is highly ill-conditioned. In other words, if the error tolerance is fixed, our new solver can deploy a (much) larger truncation error $\epsilon$ reducing the setup/factorization cost of a hierarchical solver significantly. For example, in our numerical experiments we will show that our new hierarchical solver ($\epsilon=10^{-2}$) performs better than the original LoRaSp solver ($\epsilon=10^{-4}$). In particular, $\hat{E}_{ww}$ is an order of magnitude smaller than $\tilde{E}_{ww}$ and does not depend on $\sigma_{\min}(A_{ss})$. Furthermore, $\hat{E}_{ww}$ is now symmetric positive semi-definite, which implies the following.

\begin{corollary}
The $S_{{\cal S}_s \hat{A} {\cal S}_s^T}^{(2,2)}$ block, i.e., $ww$/(2,2) block in $S_{{\cal S}_s \hat{A} {\cal S}_s^T}$ is SPD.
\end{corollary}
\begin{proof}
The following equality holds according to Eq.~(\ref{eqn:err_new}).
\[
S_{{\cal S}_s \hat{A} {\cal S}_s^T}^{(2,2)} = S_A^{(2,2)} + \hat{E}_{ww}
\]
Since the original matrix $A$ is SPD, the exact Schur complement $S_A$ and the block $S_A^{(2,2)}$ are both SPD. It is also obvious that $\hat{E}_{ww} = \hat{V}_2 \hat{V}_2^T$ is a symmetric positive semi-definite matrix. Therefore, $S_{{\cal S}_s \hat{A} {\cal S}_s^T}^{(2,2)}$ is SPD.
\hfill \end{proof}

In general, the matrix $S_{{\cal S}_s \hat{A} {\cal S}_s^T}$ itself is not necessarily an SPD matrix for any $\epsilon$. However, we observe that the matrix remains SPD for much larger $\epsilon$ (lower cost) with the deferred-compression scheme than that in the original algorithm.

Overall, the differences between computing an approximate Schur complement of $S_A$ with and without the deferred-compression scheme are summarized in the following table.

\begin{table}[htbp]
  \caption{Differences between computing an approximate Schur complement of $S_A$ with and without the deferred-compression (DC) scheme. $^*$ means corresponding blocks in the computed (approximate) Schur complements.}
  \centering
  \begin{tabular}{  c  c  c } \toprule
		&	Without DC	& With DC \\ \midrule
    \rule{0pt}{6ex} Matrix & $A$ & $\begin{pmatrix}
G_s^{-1} & &  \\
& I & \\
& & I
\end{pmatrix}
A
\begin{pmatrix}
G_s^{-T} & &  \\
& I & \\
& & I
\end{pmatrix}$ \\
   \rule{0pt}{6ex} Low rank & $A_{sw} = U V $ & 
   $G^{-1} A_{sw} = \hat{U} \hat{V}^T $ \\
   \rule{0pt}{9ex} \begin{tabular}{@{}c@{}} Approxi- \\ mation \end{tabular} & 
   $\begin{pmatrix}
A_{ss} & A_{sn} & U_1 V_1^T \\
A_{sn}^T & A_{nn} & A_{nw} \\
V_1 U_1^T & A_{nw}^T & A_{ww}
\end{pmatrix}$ &
   $\begin{pmatrix}
I & G^{-1} A_{sn} & \hat{U}_1 \hat{V}_1^T \\
(G^{-1} A_{ns})^T & A_{nn} & A_{nw} \\
\hat{V}_1 \hat{U}_1^T & A_{nw}^T & A_{ww}
\end{pmatrix}$ \\
\rule{0pt}{6ex} \begin{tabular}{@{}c@{}} Schur \\ complement \end{tabular} 
  & Eq.~(\ref{eqn:sc_old}) &  Eq.~(\ref{eqn:sc_new}) \\
\rule{0pt}{6ex} $ww$ block$^*$ & may be indefinite & always SPD  \\
   \rule{0pt}{6ex} \begin{tabular}{@{}c@{}} Error for \\ $ww$ blocks$^*$ \end{tabular}	& $2\epsilon \|A_{sw}\|/\sigma_{\min}(A_{ss}) + O(\epsilon^2)$ & $\epsilon^2$ \\
   \rule{0pt}{6ex} \begin{tabular}{@{}c@{}} Error for \\ $nw$ blocks$^*$ \end{tabular} 
   & $\epsilon \|A_{ns}\|/\sigma_{\min}(A_{ss})$ & $\epsilon \|A_{ns}\|/\sigma_{\min}(A_{ss})^{1/2}$ \\ \bottomrule
\end{tabular}
\end{table}

\section{Improved LoRaSp Solver} \label{CC:sec:algorithm}

In this section, we complete the algorithm description of our new hierarchical solver obtained by implementing the deferred-compression technique in the original LoRaSp solver. Our goal is to solve an (ill-conditioned) SPD linear system 
\begin{equation} \label{eqn:ax=b}
A x = b
\end{equation}
and our solver is based on a clustering of the unknown variables in Eq.~(\ref{eqn:ax=b}).

\paragraph{Matrix Partitioning} Define $G_A=(V, E)$ as the (undirected) graph corresponding to the symmetric matrix $A$: vertices in $G_A$ correspond to row/column indexes in $A$, and an edge $E_{p,q} = (p,q)$ exists between vertices $p$ and $q$ if $A(p,q) \not= 0$.
A clustering of unknown variables in Eq.~(\ref{eqn:ax=b}) is equivalent to a partitioning of the graph $G_A$. Graph partitioning is a well-studied problem and can be computed algebraically using techniques such as spectral partitioning and multilevel methods in existing high-performance packages, such as METIS/ParMETIS \cite{CC:karypis1998fast}, Scotch \cite{CC:chevalier2008pt} and Zoltan \cite{CC:boman2012zoltan}.

Our hierarchical solver computes an approximate factorization of $A$ by compressing fill-in blocks generated during Gaussian elimination. The key observation is that the fill-in blocks have low-rank structures, i.e., their singular values decay fast. Intuitively, the inverse of a diagonal block in the discretization matrix corresponds to the discrete Green's function of a local elliptic PDE, which have numerically off-diagonal matrix blocks. The same low-rank property also carries over to the Schur complement~\cite{CC:amestoy2015improving,CC:aminfar2016fast,CC:pouransari2017fast,CC:xia2010fast,ho2016hierarchical}. 

Below, we first illustrate applying the deferred-compression technique and the ``low-rank elimination'' step (``scaled low-rank elimination'' in the following) to one cluster of unknown variables in Section~\ref{CC:subsec:lre}. Then we present the whole algorithm in Section~\ref{CC:subsec:lorasp} and complexity analysis in Section~\ref{subsec:analysis}.

\subsection{Scaled Low-rank Elimination} \label{CC:subsec:lre}

Let $\Pi = \cup_{i=0}^{m-1} \pi_i$ denote a clustering of all unknown variables in Eq.~(\ref{eqn:ax=b}), and without loss of generality, assume that matrix $A$ is partitioned and ordered accordingly, e.g., the first block row/column corresponds to $\pi_0$. Two clusters $\pi_p$ and $\pi_q$ are defined as ``neighbors'' if the matrix block $A(\pi_p, \pi_q) \not= 0$. In other words, the neighbors of a cluster is the set of adjacent clusters in $G_A$.

To use the ``scaled low-rank elimination'' step, we partition matrix $A_0=A$ in the familiar way
\begin{equation*}
A_0 = 
\begin{pmatrix}
A_{ss} & A_{sn} & A_{sw} \\
A_{ns} & A_{nn} & A_{nw} \\
A_{ws} & A_{wn} & A_{ww}
\end{pmatrix}
\end{equation*}
where the ``s'' block corresponds to $\pi_0$, ``n'' block corresponds to neighbors of $\pi_0$ and ``w'' block corresponds to the rest. Based on our definition of neighbors above, $A_{sw} = A_{ws} = 0$. In this case and generally if $A_{sw} = A_{ws} = 0$, the ``scaled low-rank elimination'' step is reduced to normal block Cholesky factorization.

As in Section~\ref{CC:sec:diagonal_scaling}, denote ${\cal C}_s$ as the matrix corresponding to one step of block Cholesky factorization and denote $A_1$ as the Schur complement, i.e.,
\[
{\cal C}_s A_0 {\cal C}_s^T = 
\begin{pmatrix}
I & 0 \\
0 & A_1
\end{pmatrix}.
\]
Again, we can partition $A_1$ into the following $3 \times 3$ block matrix
\begin{equation*}
A_1 = 
\begin{pmatrix}
A_{ss}^{(1)} & A_{sn}^{(1)} & A_{sw}^{(1)} \\
A_{ns}^{(1)} & A_{nn}^{(1)} & A_{nw}^{(1)} \\
A_{ws}^{(1)} & A_{nw}^{(1)} & A_{ww}^{(1)}
\end{pmatrix},
\end{equation*}
where the ``s'' block corresponds to $\pi_1$, the ``n'' block corresponds to neighbors of $\pi_1$ and the ``w'' block includes all remaining vertices. Assume $A_{sw}^{(1)} \not= 0, A_{ws}^{(1)} \not= 0$, which contains fill-in generated from previous elimination of $\pi_0$. To simplify notations, we will drop the superscription of matrix blocks in $A_1$.

The ``scaled low-rank elimination'' step involves three operators: scaling operator ${\cal S}$, sparsification operator ${\cal E}$ and Gaussian elimination operator ${\cal G}$. The scaling operator ${\cal S}_s$ is defined as follows
\begin{equation} \label{CC:eqn:scaling}
{\cal S}_s = 
\begin{pmatrix}
G_s^{-1} & &  \\
& I & \\
& & I
\end{pmatrix},
\end{equation}
where $A_{ss} = G_s G_s^T$ is the Cholesky factorization.

After the scaling operator is applied, the off-diagonal block $G_s^{-1} A_{sw}$ in ${\cal S}_s A_1 {\cal S}_s^T$ is compressed with low-rank approximation, as in Eq.~(\ref{eqn:offdiag_new}). This compression step ${\cal S}_s A_1 {\cal S}_s^T \approx compress({\cal S}_s A_1 {\cal S}_s^T)$ is exactly the same as in the deferred-compression scheme. Instead of eliminating the ``s'' block directly, the next step applies the sparsification operator
\begin{equation} \label{CC:eqn:sparse}
{\cal E}_s = 
\begin{pmatrix}
\hat{U}^T & &  \\
& I & \\
& & I
\end{pmatrix}
\end{equation}
and introduces a zero block as below
\begin{align*}
{\cal E}_s \; compress({\cal S}_s A_1 {\cal S}_s^T) \; {\cal E}_s^T &=
{\cal E}_s
\begin{pmatrix}
I & G_s^{-1} A_{sn} & \hat{U}_1 \hat{V}_1^T \\
A_{ns} G_s^{-T} & A_{nn} & A_{nw} \\
\hat{V}_1 \hat{U}_1^T & A_{wn} & A_{ww}
\end{pmatrix}
{\cal E}_s^T \\
&= \begin{pmatrix}
I & & \hat{U}_1^T G_s^{-1} A_{sn} & \hat{V}_1^T \\
 & I & \hat{U}_2^T G_s^{-1} A_{sn} & 0 \\
A_{ns} G_s^{-T} \hat{U}_1 & A_{ns} G_s^{-T} \hat{U}_2 & A_{nn} & A_{nw} \\
\hat{V}_1 & 0 & A_{wn} & A_{ww}
\end{pmatrix}.
\end{align*}
Notice $\hat{U}^T \hat{U}_1 = \begin{pmatrix}
I \\ 0
\end{pmatrix}$ where the identity has the same size as the number of columns in $\hat{U}_1$, i.e., rank of the low-rank approximation in Eq.~(\ref{eqn:offdiag_new}).

After the sparsification step, a cluster of unknown variables $\pi_s$ can be split into ``coarse'' unknown variables $\pi_s^c$ and ``fine'' unknown variables $\pi_s^f$, where $\pi_s^f$ involves no fill-in. Then $\pi_s^f$ is eliminated, which does not propagate any existing fill-in (no level-2 fill-in introduced). 

The Gaussian elimination operator 
\begin{equation} \label{CC:eqn:elim}
{\cal G}_s = 
\begin{pmatrix}
I&&& \\
&I&&\\
& - A_{ns} G_s^{-T} \hat{U}_2  &I&\\
&&&I
\end{pmatrix}
\end{equation}
eliminates the ``fine'' unknown variables $\pi_s^f$ as follows
\[
{\cal G}_s {\cal E}_s \; compress({\cal S}_s A_1 {\cal S}_s^T) \; {\cal E}_s^T {\cal G}_s^T = \begin{pmatrix}
I & & \hat{U}_1^T G_s^{-1} A_{sn} & \hat{V}_1^T \\
 & I &  &  \\
A_{ns} G_s^{-T} \hat{U}_1 &  & X_{nn} & A_{nw} \\
\hat{V}_1 &  & A_{wn} & A_{ww}
\end{pmatrix},
\]
where $X_{nn} = A_{nn} - A_{ns} G_s^{-T} \hat{U}_2 \hat{U}_2^T G_s^{-1} A_{sn}$.

Last, we introduce an auxiliary permutation operator, $P_s$, to permute rows and columns corresponding to $\pi_s^c$ to the end. $P_s$ is defined as
\begin{equation}\label{CC:eqn:perm}
P_s =
\left(
\begin{array}{c|c|c|c}
&I&&  \\ \hline
&&I & \\ \hline
&&&I \\ \hline
I&&&
\end{array}
\right).
\end{equation}

Finally, define the ``scaled low-rank approximation'' operator ${\cal W}_s = P_s {\cal G}_s {\cal E}_s {\cal S}_s$ and ${\cal W}_s A_1 {\cal W}_s^T$ selects and eliminates the fine DOFs in $\pi_s$. To summarize, we have derived 
\[
{\cal W}_s A_1 {\cal W}_s^T \approx P_s {\cal G}_s {\cal E}_s \; compress({\cal S}_s A_1 {\cal S}_s^T) \; {\cal E}_s^T {\cal G}_s^T P_s^T =
\begin{pmatrix}
I & \\
& A_2
\end{pmatrix}.
\]

\subsection{Entire Algorithm} \label{CC:subsec:lorasp}


We have introduced the ``scaled low-rank elimination'' step for one cluster. The algorithm repeatedly applies this step on all clusters in $\Pi = \cup_{i=0}^{m-1} \pi_i$. This process is equivalent to computing an approximate factorization of the input SPD matrix $A$, subject to the error of low-rank approximations. After all clusters are processed, one is left with a linear system consisting of the ``coarse'' unknown variable $\cup_{i=0}^{m-1} \pi_i^c$, and we can apply the same idea on this coarse system. The entire algorithm is shown in Algorithm~\ref{CC:alg:lorasp}.

\begin{algorithm}[htbp]
  \caption{Hierarchical solver: factorization phase}
  \label{CC:alg:lorasp}   
  \begin{algorithmic}[1]
    \Procedure{Hierarchical\_Factor}{$A$}
    \If{the size of $A$ is small enough}
    \State{Factorize $A$ with the conventional Cholesky factorization}
    \State \Return
    \EndIf
    \State{Partition the graph of $A$ and obtain vertex clusters $\Pi = \cup_{i=0}^{m-1} \pi_i$} 
    \Statex \Comment{{\scriptsize $m$ is chosen to get roughly constant cluster sizes}}
    \State{$A_0 \gets A$}
    \For{$i \gets 0$ \textbf{to} $m-1$}
    \State{$A_{i+1} \gets$ Scaled\_LowRank\_Elimination($A_i$, $\Pi$, $\pi_i$)}
    \EndFor \Comment{{\scriptsize $A_m = {\cal W}_{m-1} \ldots {\cal W}_1 {\cal W}_0 \, A \, {\cal W}_0^T {\cal W}_1^T \ldots {\cal W}_{m-1}^T$}}
    \State{Extract $A_c$ from the block diagonal matrix $A_m \approx \begin{pmatrix}I&\\ & A_c \end{pmatrix}$}
    \Statex \Comment{{\scriptsize $A_c$ is the Schur complement for the coarse DOFs}}
    \State{$A_c^{fac} \gets$ \Call{Hierarchical\_Factor}{$A_c$}} 
    \Statex \Comment{{\scriptsize Recursive call with a smaller matrix}}
    \State \Return 
    $
    A^{fac} =
    {\cal W}_0^{-1} {\cal W}_1^{-1} \ldots {\cal W}_{m-1}^{-1} 
    \begin{pmatrix}
    I & \\
    & A_c^{fac}
    \end{pmatrix}
    {\cal W}_{m-1}^{-T} \ldots {\cal W}_1^{-T} {\cal W}_0^{-T}
    $
    \Statex \Comment{{\scriptsize $A_c^{fac}$ is not written out explicitly}}
    \EndProcedure
    \Statex
    \Procedure{Scaled\_LowRank\_Elimination}{$A_i$, $\Pi$, $\pi_i$}
    \State{Extract $\bar{A}$ from $A_i \approx \begin{pmatrix} I&\\ &\bar{A}\end{pmatrix}$}
    \State{Compute the low-rank elimination operator $\bar{{\cal W}}_i = P_i {\cal G}_i {\cal E}_i {\cal S}_i$ based on $\bar{A}$}
    \Statex \Comment{{\scriptsize ${\cal E}_i, {\cal G}_i \text{ and } P_i$ are defined in Eq.~\ref{CC:eqn:scaling}, Eq.~\ref{CC:eqn:elim} and Eq.~\ref{CC:eqn:perm}}}
    \State{${\cal W}_i \gets \begin{pmatrix} I&\\ & \bar{{\cal W}}_i \end{pmatrix}$}
    \Statex \Comment{{\scriptsize ${\cal W}_i$ has the same size as $A_i$}}
    \State \Return ${\cal W}_i A_i {\cal W}_i^T$
	\EndProcedure
    \Statex \Comment{{\scriptsize Notation: $a \gets b$ means assign the value $b$ to $a$, whereas $a=b$ means they are equivalent}}
  \end{algorithmic}
\end{algorithm}

Similar to sparse direct solvers, Algorithm~\ref{CC:alg:lorasp} outputs an approximate factorization of the original matrix $A$, which is used to solve the linear system $Ax=b$. Since ${\cal S}_i$ and ${\cal E}_i$ are block diagonal matrices, ${\cal G}_i$ is a triangular matrix and $P_i$ is a permutation matrix, the solve phase follows the standard forward and backward substitution, as shown in Algorithm~\ref{CC:alg:lorasp_solve}.

\begin{algorithm}[htbp]
  \caption{Hierarchical solver: solve phase}
  \label{CC:alg:lorasp_solve}   
  \begin{algorithmic}[1]
    \Procedure{Hierarchical\_Solve}{$A^{fac}$, $b$}
    \State{$y \gets$ \Call{Forward\_Substitution}{$A^{fac}$, $b$}}
    \State{$x \gets$ \Call{Backward\_Substitution}{$A^{fac}$, $y$}}
    \State \Return $x$
    \EndProcedure
    \Statex
    \Procedure{Forward\_Substitution}{$A^{fac}$, $b$}
    \State{$y \gets b$}
    \For{$i \gets 0$ \textbf{to} $m-1$}
    \State{$y \gets {\cal W}_i \, y$ \Comment{{\scriptsize $y$ is overwritten}}}
    \EndFor \Comment{{\scriptsize $y = (y_c, y_f)$ is of the concatenation of $y_f$ and $y_c$}}
    \State{Extract $y_f$ and $y_c$ from $y$}
    \Statex \Comment{{\scriptsize $y_f$ and $y_c$ correspond to the fine DOFs and the coarse DOFs}}
    \State{$y_c \gets$ \Call{Forward\_Substitution}{$A_c^{fac}$, $y_c$} \Comment{{\scriptsize $y_c$ is overwritten}}}
    \State \Return $y = (y_f, y_c)$ \Comment{{\scriptsize output the concatenation of $y_f$ and $y_c$}}
    \EndProcedure
    \Statex    
    \Procedure{Backward\_Substitution}{$A^{fac}$, $y$}
    \State{$x \gets y$}
    \For{$i \gets m-1$ \textbf{to} $0$}
    \State{$x \gets {\cal W}_i^T \, x$ \Comment{{\scriptsize $x$ is overwritten}}}
    \EndFor \Comment{{\scriptsize $x = (x_f, x_c)$ is of the concatenation of $x_c$ and $x_f$}}
    \State{Extract $x_f$ and $x_c$ from $x$}
    \Statex \Comment{{\scriptsize $x_f$ and $x_c$ correspond to the fine DOFs and the coarse DOFs}}
    \State{$x_c \gets$ \Call{Backward\_Substitution}{$A_c^{fac}$, $x_c$} \Comment{{\scriptsize $x_c$ is overwritten}}}
    \State \Return $x = (x_f, x_c)$ \Comment{{\scriptsize Output the concatenation of $x_f$ and $x_c$}}
    \EndProcedure
    \Statex \Comment{{\scriptsize Notation: $a \gets b$ means assign the value $b$ to $a$, whereas $a=b$ means they are equivalent}}
  \end{algorithmic}
\end{algorithm}

\subsection{Complexity Analysis} \label{subsec:analysis}

The computational cost and memory requirement of the original LoRaSp method and the corresponding parallel algorithm are analyzed in \cite{CC:pouransari2017fast} and \cite{CC:chen2018hierarchical}, respectively. A key assumption of these analyses is that ranks of the low-rank truncations can be bounded from above. We observe this in practice, but it is not possible to guarantee this boundedness without making additional hypotheses on the input matrix. The behavior of ranks in hierarchical matrices has been studied in several existing papers~\cite{bebendorf2003existence,bebendorf2005efficient,chandrasekaran2010numerical}. Here, we make a similar assumption to earlier works on the boundedness of ranks, which are based on ideas concerning the underlying Green's function that are related to standard multipole estimates~\cite{greengard1987fast,greengard1997new}.

Below we rephrase Theorem 5.4 in~\cite{CC:pouransari2017fast} and state that it holds as well for the improved LoRaSp solver when similar assumptions are made as in~\cite{CC:pouransari2017fast}. Complexity analysis of the corresponding parallel algorithm is summarized in  Theorem~\ref{thm:parallel}, which is again a rephrase of results in~\cite{CC:chen2018hierarchical}. Note that memory and the solve time have the same complexity. Intuitively, the solve phase touches every nonzero once.

\begin{theorem} \label{thm:linear}
In the (improved) LoRaSp algorithm, the computational cost of the factorization is $O(N r^2)$, and the computational cost of the solve (per iteration) and the memory consumption both scale as $O(N r)$, where $N$ is the problem size and $r$ is the largest cluster size at the first/finest level (level 0), if the following two conditions hold:
\begin{enumerate}
\item for every cluster of unknown variables, the number of neighbor clusters is bounded by a constant.
\item the largest cluster size at the first level (level 0), $r$, is bounded by a constant; and $r_i$, the largest cluster size at level $i$, satisfies the relationship that $r_i \le \alpha^i \, r$, where $0 < \alpha < 2^{1/3}$.
\end{enumerate}
\end{theorem}

\begin{theorem} \label{thm:parallel}
Assume the linear system is evenly distributed among all processors, the conditions in Theorem~\ref{thm:linear} hold and all clusters have $r$ unknown variables. The computational cost of the factorization and the solve (same as the memory consumption) are $O(N r^2 / p)$ and $O(N r / p)$ on every processor, where $p$ is the number of processors. Further, for every processor, the amount of communication is 
\[
O\bigg( r^2 \; \big(\frac{N}{rp} \big)^{2/3} \bigg) 
= O\bigg( \Big(\frac{N r^2}{p} \Big)^{2/3} \bigg)
\]
for a 3D underlying subdomain, and the number of messages sent by every processor is 
\[
O\bigg(\log\big(\frac{N}{rp}\big)\bigg) + O(\log p).
\]
\end{theorem}

\section{Ice Sheet Model} \label{CC:sec:stokes}
We focus on the first-order Stokes model~\cite{CC:tezaur2015albany}. This simplified model preserves sufficient accuracy for simulating the flow over most parts of an ice sheet and is computationally attractive when compared to a full Stokes model. The ice sheet model is discretized with a Galerkin finite element method using either bilinear or trilinear basis functions on tetrahedral or hexahedral elements, respectively. Further details of that underlying discretization can be found in~\cite{CC:tezaur2015albany, CC:tezaur2015scalability}. Below, we provide details on the partial differential equation (PDE) and the corresponding boundary conditions.

\subsection{Stokes Formulation and Discretization}

The goal of an ice sheet model is to solve for the $x$ and $y$ components of the ice velocity. These two components are approximated by the following elliptic system of PDEs:
\begin{equation} \label{CC:eqn:stokes}
\begin{cases}
& - \nabla \cdot (2u \dot{\boldsymbol{\epsilon}}_1) + \rho g \frac{\partial s}{\partial x} = 0 \\
& - \nabla \cdot (2u \dot{\boldsymbol{\epsilon}}_2) + \rho g \frac{\partial s}{\partial y} = 0
\end{cases}
\end{equation}
where $\mu$ is the “effective” viscosity, $\rho$ is ice density,  $g$ is the gravitational acceleration, and $s \equiv s(x, y)$ denotes the upper boundary surface. The $\boldsymbol{\epsilon}_i$ are approximations to the effective strain rate tensors:
\begin{equation}
\dot{\boldsymbol{\epsilon}}_1^T = (2 \dot{\epsilon}_{xx} + \dot{\epsilon}_{yy}, \dot{\epsilon}_{xy}, \dot{\epsilon}_{xz}) \quad \text{and} \quad \dot{\boldsymbol{\epsilon}}_2^T = (  \dot{\epsilon}_{xy}, \dot{\epsilon}_{xx} + 2\dot{\epsilon}_{yy}, \dot{\epsilon}_{yz}) 
\end{equation}
where
\begin{equation}
\dot{\epsilon}_{xx} = \frac{\partial u}{\partial x}, \quad \dot{\epsilon}_{yy} = \frac{\partial v}{\partial y}, \quad \dot{\epsilon}_{xy} = \frac{1}{2}(\frac{\partial u}{\partial x} + \frac{\partial v}{\partial y}), \quad \dot{\epsilon}_{xz} = \frac{1}{2} \frac{\partial u}{\partial z}, \quad \dot{\epsilon}_{yz} = \frac{1}{2} \frac{\partial v}{\partial z} .
\end{equation}
Nonlinearity arises from the “effective” viscosity, which is approximated by
\begin{equation}
\mu = \frac{1}{2} A^{-\frac{1}{n}} \; \dot{\epsilon}_e^{-\frac{2}{n}},
\end{equation}
using \textit{Glen's law}~\cite{CC:cuffey2010physics, CC:nye1957distribution} to model the ice rheology. Here, $\dot{\epsilon}_e$ is the effective strain rate given by
\begin{equation}
\dot{\epsilon}_e^2 \equiv \dot{\epsilon}_{xx}^2 + \dot{\epsilon}_{yy}^2 + \dot{\epsilon}_{xx}\dot{\epsilon}_{yy} + \dot{\epsilon}_{xy}^2 + \dot{\epsilon}_{xz}^2 + \dot{\epsilon}_{yz}^2
\end{equation} 
and $A$ is a temperature-dependent factor that can be described through an Arrhenius relation~\cite{CC:cuffey2010physics}. 
In this work, we take $n = 3$, as is commonly done. A combination of Newton's method and continuation generates a sequence of linear systems for the new hierarchical solver. 


On the top boundary, a homogeneous Neumann condition is prescribed: $\dot{\boldsymbol{\epsilon}}_1 \cdot \boldsymbol{n} = \dot{\boldsymbol{\epsilon}}_2 \cdot \boldsymbol{n} = 0$, where $\boldsymbol{n}$ is the outward facing normal vector to the upper surface. On the bottom boundary, a Robin condition is used:
\begin{equation} \label{CC:eqn:bc}
\begin{aligned}
2 \mu \dot{\boldsymbol{\epsilon}}_1 \cdot \boldsymbol{n} + \beta u = 0 \\
2 \mu \dot{\boldsymbol{\epsilon}}_2 \cdot \boldsymbol{n} + \beta v = 0 
\end{aligned}
\end{equation}
where $\beta \equiv \beta (x,y) \ge 0$ is the basal sliding (or friction) coefficient that in this paper can be viewed as an already known field.
Large $\beta$ (e.g., $\beta = 10^4$ kPa yr m$^{-1}$) corresponds to a quasi-no-slip condition, while small $\beta$ implies a weak frictional force, corresponding to a thawed ice-bed interface that allows for some degree of slip tangential to the bedrock. Under floating ice shelves, $\beta$ is often taken as identically equal to zero, corresponding to a frictionless boundary. Fig.~\ref{CC:fig:partition} (left) shows the distribution of $\beta$ in Antarctica. On the lateral boundary, a dynamic Neumann condition (referred to as “open-ocean” or “floating ice”) is used:
\begin{equation}
\begin{aligned}
2 \mu \dot{\boldsymbol{\epsilon}}_1 \cdot \boldsymbol{n} - \rho g (s-z) \boldsymbol{n} = \rho_w g \max(z, 0) \boldsymbol{n} \\
2 \mu \dot{\boldsymbol{\epsilon}}_2 \cdot \boldsymbol{n} - \rho g (s-z) \boldsymbol{n} = \rho_w g \max(z, 0) \boldsymbol{n}
\end{aligned}
\end{equation}
where $\rho_w$ denotes the density of water and $z$ is the elevation above sea level. This condition is derived under a hydrostatic equilibrium assumption between the ice shelf and the air (or water) that surrounds it~\cite{CC:macayeal1996ice}.

\subsection{Extruded Partitioning for Ice Sheets} \label{CC:sec:partition}

The improved LoRaSp method is based on an extruded partitioning of a three-dimensional extruded mesh, which logically corresponds to a tensor product of a two-dimensional unstructured mesh in the $x$, $y$ directions with a one-dimensional mesh in the $z$ direction. Specifically, one layer of the three-dimensional extruded mesh, i.e., an unstructured two-dimensional mesh, is partitioned using a general graph partitioner, such as METIS/ParMETIS~\cite{CC:karypis1998fast}, Scotch/PT-scotch~\cite{CC:chevalier2008pt}, and Zoltan~\cite{CC:boman2012zoltan}; the partitioning result is then extruded in the third dimension such that mesh vertices lying on the same extruded line always belong to the same cluster. The motivation of our extruded partitioning scheme is that a mesh point is closer to its vertical neighbors than its horizontal neighbors because vertical coupling is stronger than its horizontal counterpart in ice sheets modeling. Fig.~\ref{CC:fig:partition} (right) shows the partitioning result of the mesh used for Antarctic ice sheet modeling (the extruded dimension is not shown).

\begin{figure}[htbp]
\begin{center}
\caption{Antarctic ice sheet modeling: (left) distribution of the basal sliding coefficient; (right) partitioning of the two-dimensional mesh, i.e, one layer of the three-dimensional extruded mesh.}
\label{CC:fig:partition}
\end{center}
\end{figure}

Note the extruded partitioning scheme does not assume that the mesh spacing in the extruded dimension is uniform, or that mesh vertexes residing on the same mesh layer have the same $z$ coordinate value. For a number of practical reasons, vertically extruded meshes are commonly employed in ice sheet modeling. 
In addition to the use in ice sheet modeling, extruded meshes are also heavily used in other geophysical modeling applications (e.g., atmospheric and oceanic, oil/gas, carbon sequestration) and arise frequently in engineering simulations involving thin structures.


\section{Numerical Results} \label{CC:sec:results}

This section demonstrates the efficiency and the (parallel) scalability of our hierarchical solver. In particular, we want to answer the following two questions:
\begin{enumerate}
\item how does the computation costs, including factorization cost, solve cost per iteration and number of iterations, increase as the problem size increases?
\item how does the running time (factorization cost + solve cost per iteration $\times$ iteration count) of our hierarchical solver compare with that of other state-of-the-art methods?
\end{enumerate}~\\ 
\vspace{-0.7cm}
\paragraph{Test problems} We show results for solving linear systems arising from simulating ice sheets on Antarctica. These simulations are carried out on a sequence of increasingly large meshes corresponding to horizontal refinement (fixed number of vertical layers), as is commonly done in practice. The linear systems are solved using the (right) preconditioned GMRES (a restarted GMRES(200) from the Trilinos Belos\footnote{https://trilinos.org/packages/belos/} package) with a stopping tolerance of $10^{-12}$ and a maximum number of iterations of 1,000. 

\paragraph{Parameters in hierarchical solver} In our hierarchical solver, partitions are computed using geometric coordinates of mesh grids by calling the Zoltan~\cite{CC:boman2012zoltan} library, with cluster sizes around 100, which empirically gives good performance. The (only) other parameter $\epsilon$, i.e., errors of low-rank approximations, is varied to show trade-off between the costs of factorization and solve. When $\epsilon$ decreases (more accurate approximations), the factorization cost increases and the number of preconditioned iterations decreases.

\paragraph{Machine} All experiments were run on the NERSC Edison (Cray XC30) supercomputer\footnote{http://www.nersc.gov/users/computational-systems/edison/}, where every compute node has two 12-core Intel ``Ivy Bridge" processors at 2.4 GHz, and nodes are connected with Cray Aries with Dragonfly topology. Our parallel hierarchical solver is implemented using C++ and MPI. The code is compiled with icpc (ICC) 18.0.1 and linked with the Intel MKL library.

\subsection{Improved efficiency} \label{subsec:ice_sheet}

This subsection shows the improved efficiency of the hierarchical solver with the vertical partitioning step and the deferred-compression scheme. The focus is on the number of iterations because if we assume the factorization time and the solve time per iteration of the hierarchical solver are both $\mathcal{O}(N)$, then the total running time only depends on the iteration count. The four test problems used in this subsection are the following.

\begin{table}[htbp]
  \caption{Four test matrices used in Section~\ref{subsec:ice_sheet}}
  \centering
  \begin{threeparttable}
  \begin{tabular}{  c  c  c } \toprule
    $h$ & $N$ & \# of vertical mesh layers \\ \midrule
    64km  & 63,126    & 9 \\ 
    32km  & 245,646   & 9 \\ 
    16km  & 969,642     & 9 \\ 
    8km   & 3,848,868    & 9 \\ \bottomrule
  \end{tabular}
\begin{tablenotes}
\small 
\item $h$: horizontal mesh resolution/spacing
\item $N$: number of unknown variables.
\end{tablenotes}
\end{threeparttable}
\end{table}

\paragraph{Original LoRaSp method}
We first show the poor performance of the original LoRaSp solver, if applied directly to the smallest test matrix corresponding to a resolution of 64km between adjacent mesh points. In the original solver, matrix partitioning is computed algebraically with hypergraph partitioning~\cite{devine2006parallel} based on the sparsity of the discretization matrix, which ignores the numerical values in the matrix and would not capture the underlying the strong/weak coupling. Although more sophisticated partitioning algorithms, which assign matrix entries to edge weights in the adjacency graph, may lead to better partitioning results, it is beyond the scope of this paper to explore such effects.

As Table~\ref{CC:table:raw} shows, the original LoRaSp solver did not converge in 100 iterations when $\epsilon \le 10^{-3}$; when $\epsilon = 10^{-4}$, the solver converged at 69 iterations with a significant computation time (as compared to results in Table \ref{CC:table:partition}). 

\begin{table}[!htb]
\caption{Original LoRaSp solver applied to the linear system corresponding to 64km resolution.}
\label{CC:table:raw}
\centering
\begin{threeparttable}
  \begin{tabular}{  c  c  c  c  c } \toprule
  $\epsilon$ & $10^{-1}$ & $10^{-2}$ & $10^{-3}$ & $10^{-4}$  \\ \midrule
  Factor (s)  & 12  & 31  & 85  &   134  \\ 
  Solve (s)   & ---  & ---  & ---  &   44  \\ 
  Iter \#   & $100^a$  & $100^a$  & $100^a$  &   69  \\ 
  Memory (GB)  & 1  & 3  & 6  &   8  \\ \bottomrule
  \end{tabular}
  \begin{tablenotes}
  \small \item
  $^a$ solver didn't converge in 100 iterations.
  \end{tablenotes}
\end{threeparttable}
\end{table}

\paragraph{Extruded partitioning}
Table~\ref{CC:table:partition} shows the factorization time, the solve time (for all iterations), iteration number and the storage cost of the hierarchical solver using the extruded partitioning scheme. With a pre-processing step of vertical partitioning, the original LoRaSp solver becomes much more efficient for solving problems from ice sheet modeling. For example, comparing the first column in Table~\ref{CC:table:partition} with the last column in Table \ref{CC:table:raw}, we see that the total time is about 1 second and 178 seconds for 64km, respectively. Although the performance of LoRaSp has improved significantly with vertical partitioning, the number of iterations doubles as the mesh is refined as shown in Table~\ref{CC:table:partition}. Suppose the factorization time and the solve time per iteration of the hierarchical solver are both $\mathcal{O}(N)$, the total running time is $\mathcal{O}(N^{3/2})$ as the iteration number grows as $\mathcal{O}(N^{1/2})$.

\begin{table}[H]
  \caption{Extruded partitioning. Fixed $\epsilon=10^{-1}$ in LoRaSp solver.}
  \label{CC:table:partition}
  \centering
  \begin{tabular}{  c  c  c  c  c } \toprule
    Resolution & 64km & 32km & 16km & 8km  \\ \midrule
    Factor (s)  & 0.67  & 2.5  & 10  &  41  \\ 
    Solve (s)   & 0.41  & 3.7  & 33  &  220  \\ 
    Iter \#     & 12    & 26   & 52  &  107  \\ 
    Memory (GB) & 0.4   & 4    & 7   &  27  \\ \bottomrule
  \end{tabular}
\end{table}

Table~\ref{table:partition_steps} shows the number of iterations of different values of $\epsilon$ for increasing problem sizes. As shown in the table, the number of iterations decreases as $\epsilon$ decreases. When $\epsilon \ge 10^{-3}$, the iteration number roughly doubles as meshes are refined. When $\epsilon = 10^{-4}$, the number of iterations increases relatively slowly. In principle, we could further decrease $\epsilon$ and the number of iterations would be further reduced. But the increase of factorization cost with a smaller $\epsilon$ may lead to a higher total running time.

\begin{table}[htbp]
  \caption{Extruded partitioning. Number of iterations for different values of $\epsilon$.} \label{table:partition_steps}
  \centering
  \begin{tabular}{  c  c  c  c  c } \toprule
    $\epsilon$ & $10^{-1}$ & $10^{-2}$ & $10^{-3}$ & $10^{-4}$  \\ \midrule
    64km & 12	& 12	& 11	& 11 \\
	32km &	26	& 22	& 21	& 17 \\
	16km &	52	& 44	& 37	& 28 \\
	8km	 &	107 & 83	& 71	& 35 \\ \bottomrule
\end{tabular}
\end{table}

\paragraph{Deferred compression} Table~\ref{table:ds} shows the number of iterations of different values of $\epsilon$ when the deferred-compression scheme is used. As the table shows, the number of iterations is reduced significantly. More importantly, the iteration count is almost constant when $\epsilon \le 10^{-3}$ and increases logarithmically when $\epsilon = 10^{-2}$. Suppose the factorization time and the solve time per iteration are both $\mathcal{O}(N)$, the total running time would be $\mathcal{O}(N)$ or $\mathcal{O}(N\log(N))$ when $\epsilon \le 10^{-3}$ or $\epsilon = 10^{-2}$, respectively.

\begin{table}[htbp]
  \caption{Deferred-compression scheme. Number of iterations for different values of $\epsilon$.} \label{table:ds}
  \centering
  \begin{tabular}{  c  c  c  c  c } \toprule
    $\epsilon$ & $10^{-1}$ & $10^{-2}$ & $10^{-3}$ & $10^{-4}$  \\ \midrule
    64km &  14  & 10    & 9	& 5 \\
	32km &	21	& 12	& 8	& 5 \\
	16km &	37  & 14	& 8	& 5 \\
	8km	 &	54	& 16	& 8	& 6 \\ \bottomrule
\end{tabular}
\end{table}

\subsection{Ice sheet problems} \label{subsec:ilu}

In this subsection, we show running time of our hierarchical solver for solving practically large-scale linear systems from ice sheet modeling. 
Based on previous results, we chose $\epsilon = 10^{-2}$ for our hierarchical solver, which incorporates the deferred-compression scheme and extruded partitioning. 

Our reference method is the ILU-preconditioned domain decomposition method used in the Albany package~\cite{CC:tezaur2015albany} developed at the Sandia National Laboratories for ice sheet modeling. The ILU module is a well-tuned high-performance implementation in Trilinos IFPACK\footnote{https://trilinos.org/packages/ifpack/}. Since the factorization time of ILU is a tiny fraction in the total runtime, it is not shown explicitly in the following figures and tables.

In the following numerical experiments, we will fix the number of vertical mesh layers at either 6 mesh layers or 11 mesh layers, which are two common ice sheet modeling choices for low- and high-accuracy. Correspondingly, the numbers of unknowns on the same vertical line are 12 and 22 (as there are two unknowns associated with every grid point).

\paragraph{6 vertical mesh layers} Fig.~\ref{fig:ilu} shows the total running time of a weak scaling experiment\footnote{the problem size increases proportionally to the number of processors used. In other words, the problem size per processor is fixed.}, where a sequence of problems are solved on 1, 4, 15, 64 and 256 processors. As Fig.~\ref{fig:ilu} (left) shows, the running time of ILU blows up as the problem size increases, while that of the hierarchical solver remains almost constant. Fig.~\ref{fig:ilu} (right) shows the decay of residuals, and the convergence of ILU deteriorates significantly as the problem size increases.

\begin{figure}[htbp]
\begin{center}
\scalebox{0.33}{\includegraphics[natwidth=610,natheight=642]{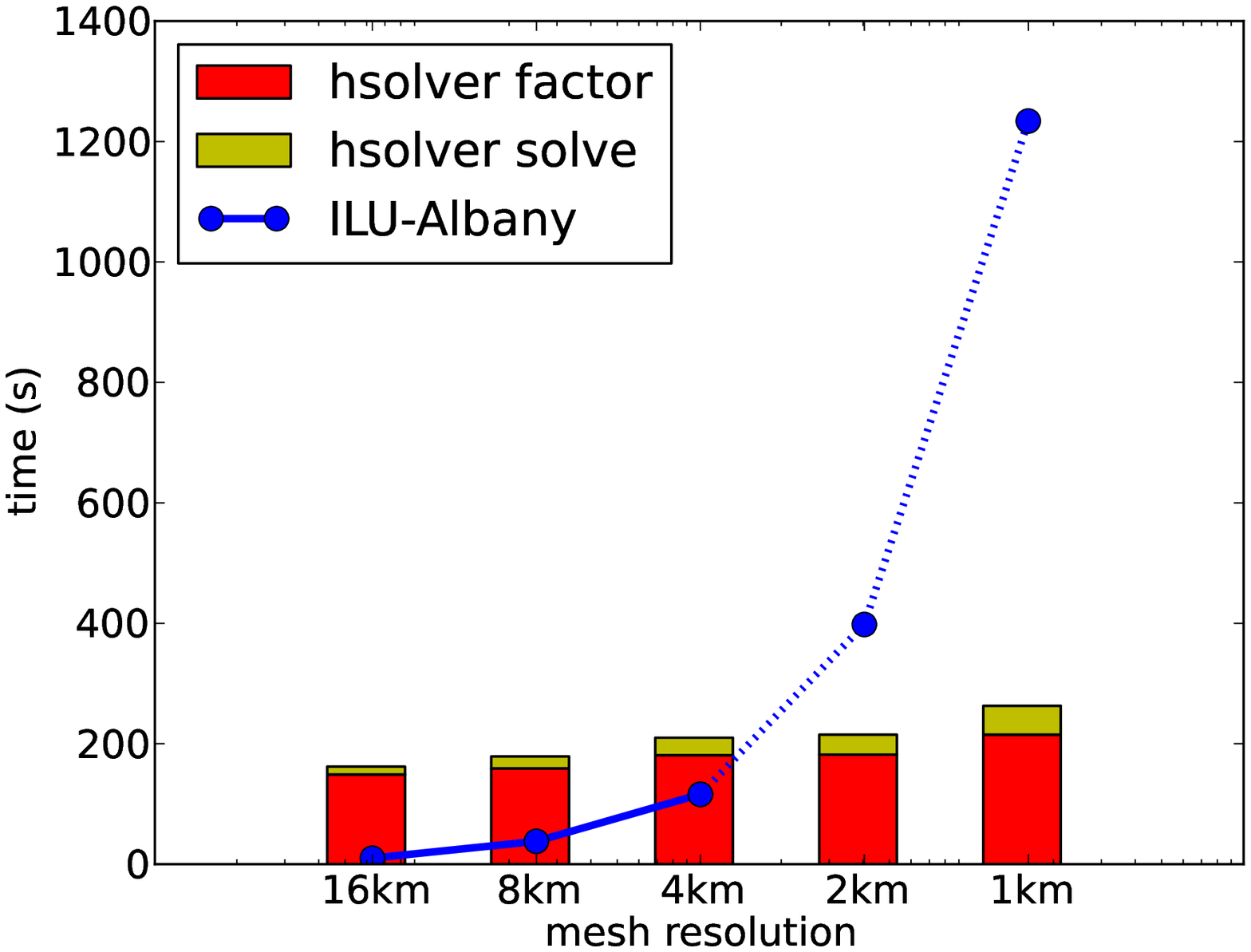}} 
\scalebox{0.33}{\includegraphics[natwidth=610,natheight=642]{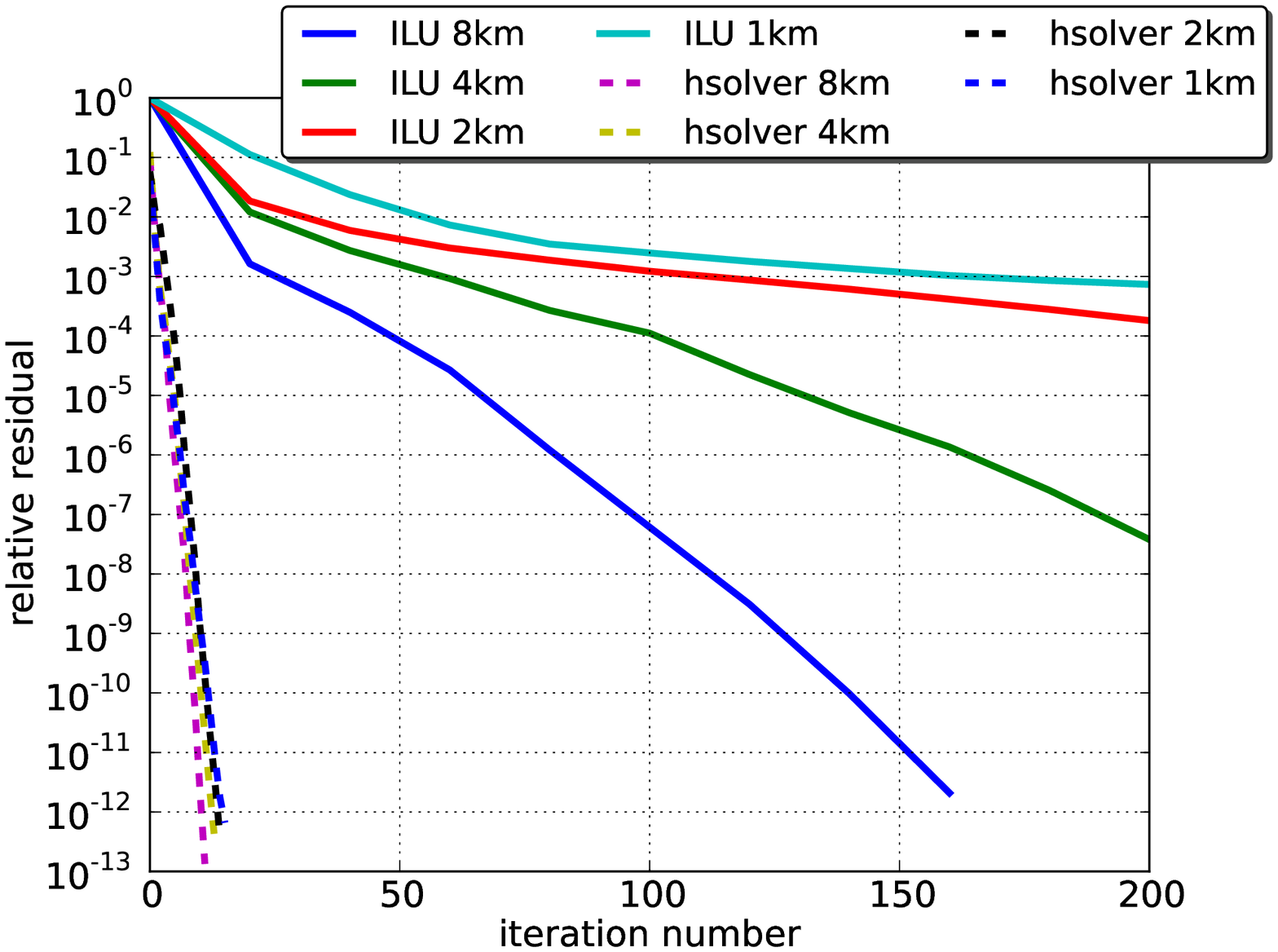}} 
\caption{6 vertical mesh layers. Weak scaling experiment on 1, 4, 16, 64, 256 processors. (Left) comparison of the total runtime (factorization+solve) between ILU (factorization time is negligible and not shown explicitly) and our hierarchical solver (hsolver). Dashed line means extrapolation based on existing data because ILU didn't converge to $10^{-12}$. (Right) Decay of residuals in ILU and our hierarchical solver (hsolver).}
\label{fig:ilu}
\end{center}
\end{figure}

\begin{table}[H]
  \caption{6 vertical mesh layers: hierarchical solver ($\epsilon=10^{-2}$) vs. ILU.} \label{table:ilu}
  \centering
  \begin{threeparttable}
  \begin{tabular}{  c  c  c | c  c | c  c  c} \toprule
     & &  & \multicolumn{2}{|c}{ILU} & \multicolumn{3}{|c}{hierarchical solver}  \\ 
    $h$ & $N$ & $P$ & iter \# & total time & iter \# & factor & solve \\ \midrule
    16km &  629K & 1    & 64    & 10	& 10 & 149 & 13 \\
	8km &	2.5M & 4	& 170	& 38	& 12 & 159 & 20 \\
	4km &	10M  & 16   & 498	& 116	& 14 & 181 & 29 \\
	2km	&	40M  & 64	& $1000^a$	& --- & 14 & 182 & 33 \\
    1km	&	161M & 256	& $1000^b$  & --- & 15 & 215 & 48 \\\bottomrule
\end{tabular}
\begin{tablenotes}
\small 
\item $h$: horizontal mesh resolution/spacing, $N$: number of unknown variables, $P$: number of processors.
\item
$^a$ ILU didn't converge to $10^{-12}$; it took 398 seconds for 1000 iterations (residual $\approx 10^{-10}$).
\item
$^b$ ILU didn't converge to $10^{-12}$; it took 346 seconds for 1000 iterations (residual $\approx 10^{-6}$).
\end{tablenotes}
\end{threeparttable}
\end{table}

Detailed information about this weak scaling experiment is summarized in Table~\ref{table:ilu}. As the mesh is refined every time, the number of iterations for ILU doubles, whereas it increases by only one or two steps for our hierarchical solver. As a result, we conclude that the computation cost of ILU is $\mathcal{O}(N^{3/2})$ as the iteration count increases as $\mathcal{O}(N^{1/2})$ empirically. By contrast, our hierarchical solver achieved $\mathcal{O}(N\log(N))$ computational complexity.

\begin{figure}[t]
\begin{center}
\scalebox{0.31}{\includegraphics[natwidth=610,natheight=642]{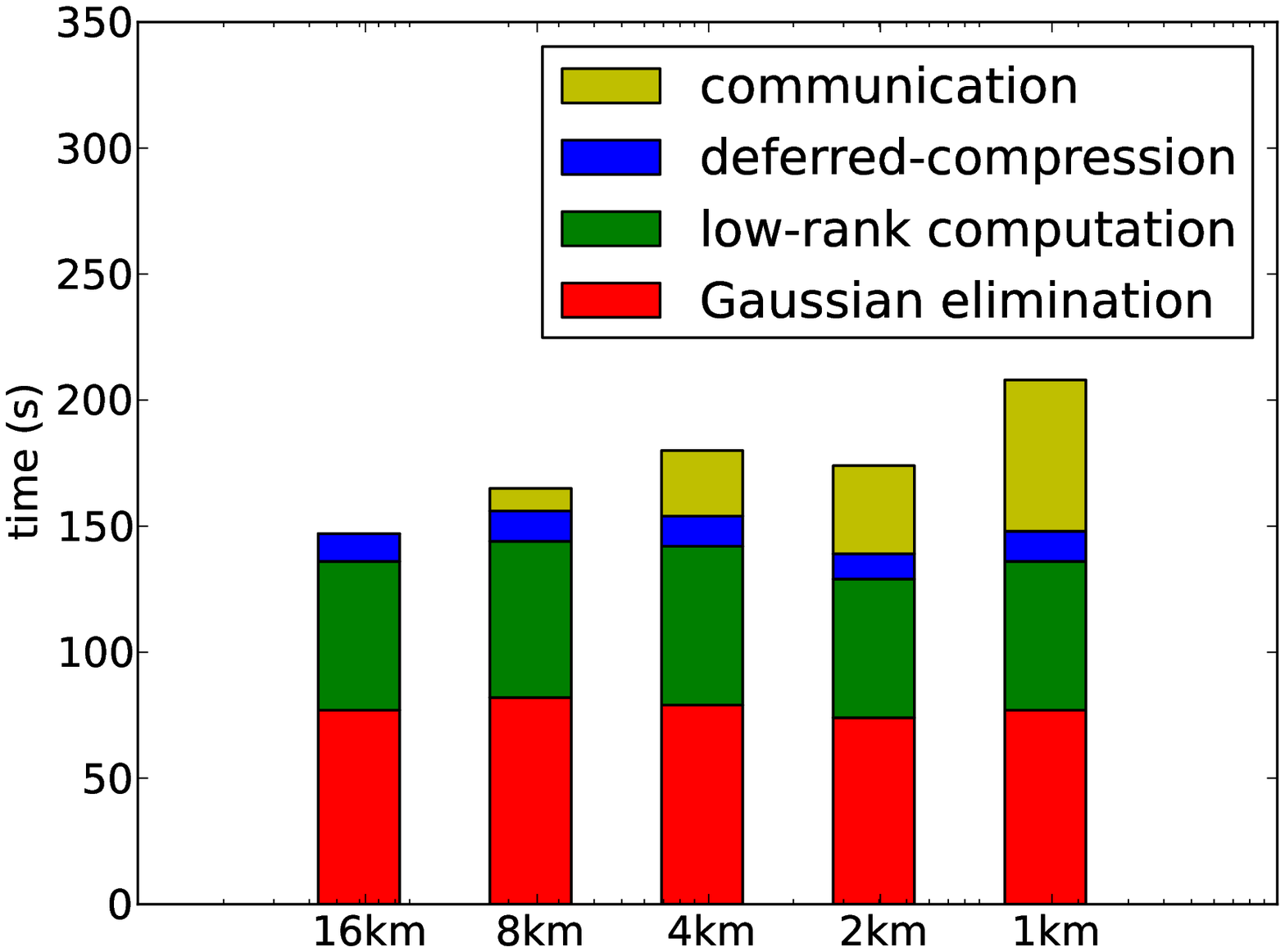}} 
\scalebox{0.31}{\includegraphics[natwidth=610,natheight=642]{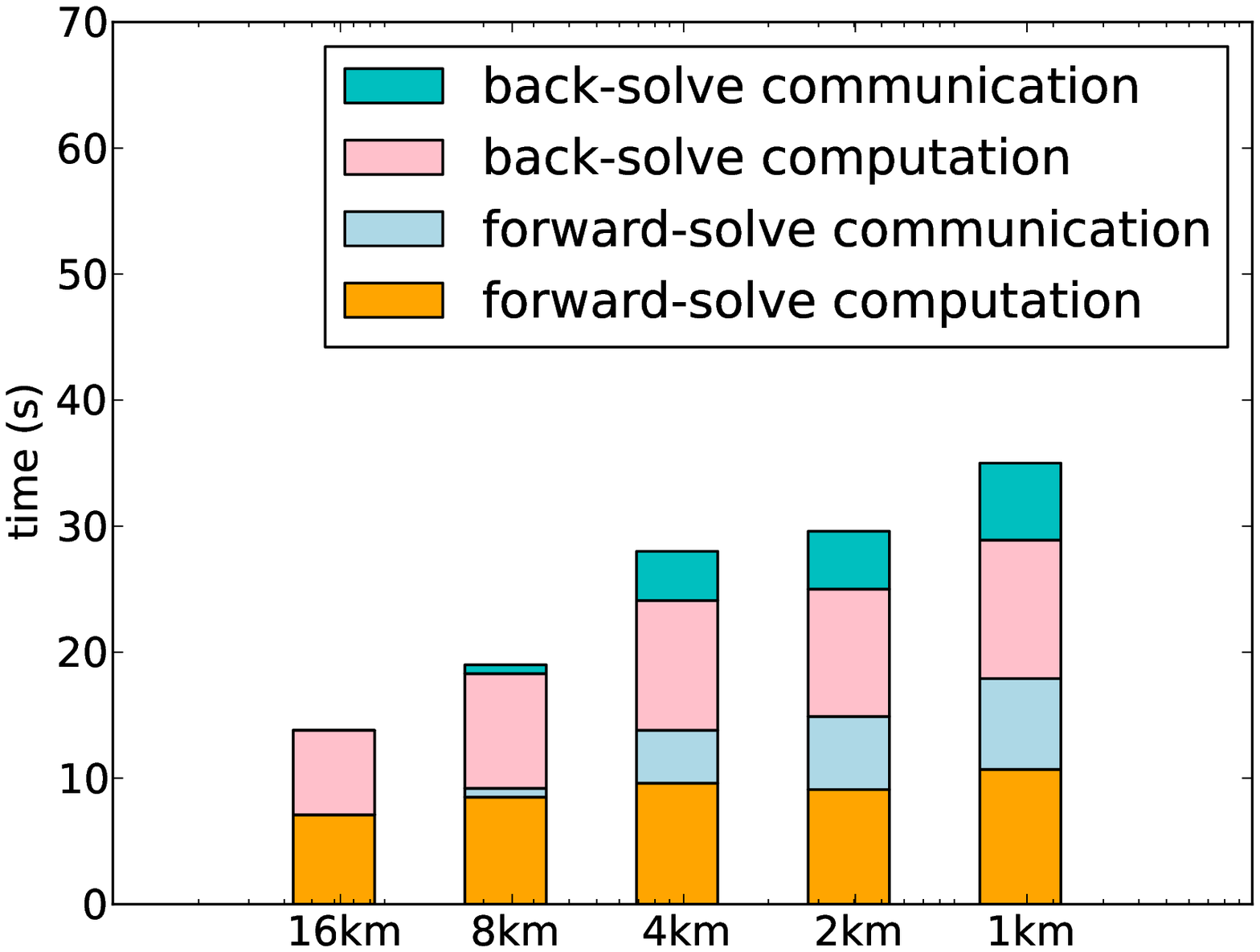}} 
\scalebox{0.37}{\includegraphics[natwidth=610,natheight=642]{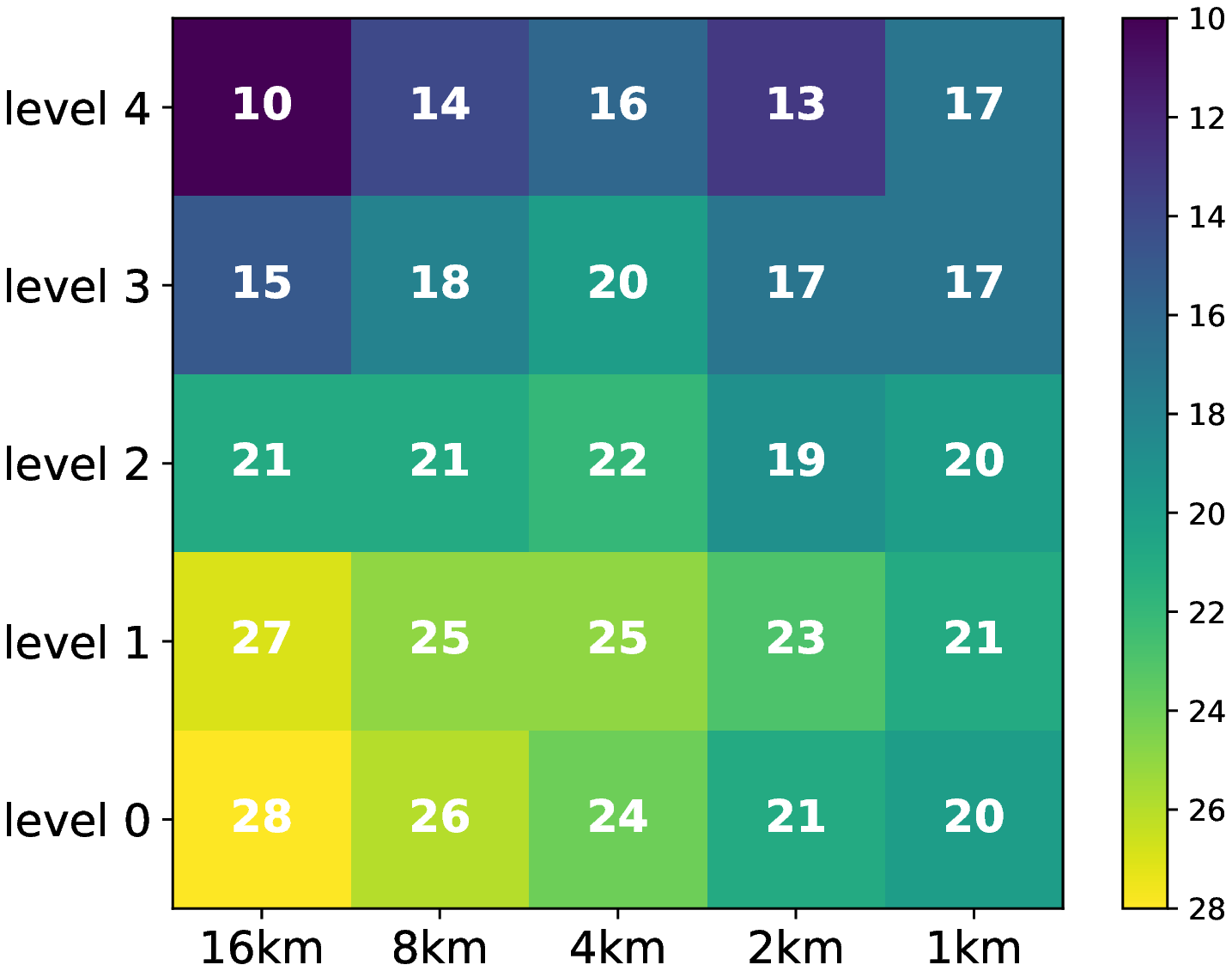}} 
\caption{6 vertical mesh layers. Weak scaling experiment on 1, 4, 16, 64, 256 processors. (First row) breakdown of the factorization time and the solve time (forward substitution+backward substitution for all iterations) in Table~\ref{table:ilu}. (Second row) Average sizes of low-rank compression at different levels (``level 0'' is the finest level) in the hierarchical solver ($\epsilon=10^{-2}$).}
\label{fig:breakdown}
\end{center}
\end{figure}

Fig.~\ref{fig:breakdown} (first row) shows the breakdown of the factorization time and the solve time (for all iterations) on one processor in the parallel hierarchical solver. In our weak scaling experiment, the deferred-compression time, low-rank compression time, Gaussian elimination time and solve time all stay almost constant as the problem size increases (proportionally to the number of processors used). Moreover, the cost of the deferred-compression scheme is only a small fraction of the total factorization cost.

Fig.~\ref{fig:breakdown} (second row) shows the average sizes of low-rank compression at all levels and problem sizes are well-bounded and hence the total running time of the hierarchical solver scales closely to $\mathcal{O}(N)$ as Thm.~\ref{thm:linear} and Thm.~\ref{thm:parallel} state.

\paragraph{11 vertical mesh layers} A weak scaling study for solving a sequence of increasingly large linear systems on 4, 16, 64, 256 and 1024 processors are shown in Table~\ref{table:10layers}. Again, the number of iterations of ILU increases as $\mathcal{O}(N^{1/2})$ while that of the hierarchical solver increases very slowly. As a result, the computation cost of ILU behaves as $\mathcal{O}(N^{3/2})$, whereas our hierarchical solver scales as $\mathcal{O}(N\log(N))$.

\begin{table}[htbp]
  \caption{11 vertical mesh layers: hierarchical solver ($\epsilon=10^{-2}$) vs. ILU.} \label{table:10layers}
  \centering
  \begin{threeparttable}
  \begin{tabular}{  c  c  c | c  c | c  c  c} \toprule
     &  &  & \multicolumn{2}{|c}{ILU} & \multicolumn{3}{|c}{hierarchical solver}  \\ 
    $h$ & $N$ & $P$ & iter \# & total time & iter \# & factor & solve \\ \midrule
    16km &  1.1M	& 4	    & 90    & 7	    & 18 & 147 & 22 \\
	8km &	4.6M	& 16	& 183	& 21	& 23 & 186 & 38 \\
	4km &	18.5M	& 64 	& 468	& 66	& 24 & 213 & 53 \\
	2km	&	74M		& 256	& $1000^a$	& ---	& 27 & 214 & 65 \\
    1km	&	296M	& 1024	& $1000^b$  & ---	&  27 & 243 & 71 \\\bottomrule
\end{tabular}
\begin{tablenotes}
\small 
\item $h$: horizontal mesh resolution/spacing, $N$: number of unknown variables, $P$: number of processors. 
\item $^a$ ILU didn't converge to $10^{-12}$; it took 145 seconds for 1000 iterations (residual $\approx 10^{-9}$). 
\item $^b$ ILU didn't converge to $10^{-12}$; it took 83 seconds for 1000 iterations (residual $\approx 10^{-3}$).
\end{tablenotes}
\end{threeparttable}
\end{table}

\begin{figure}[htbp]
\begin{center}
\scalebox{0.31}{\includegraphics[natwidth=610,natheight=642]{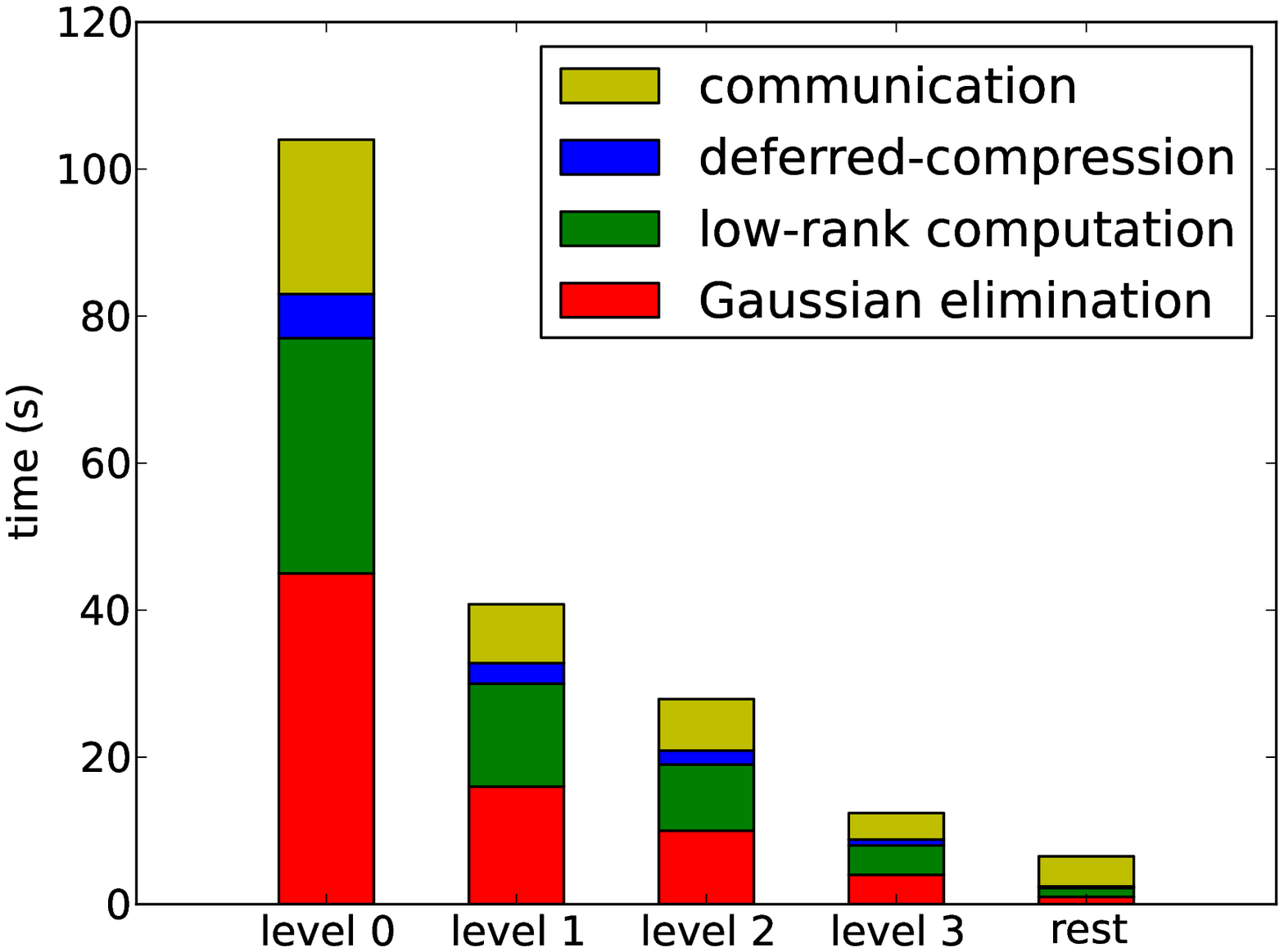}} 
\scalebox{0.31}{\includegraphics[natwidth=610,natheight=642]{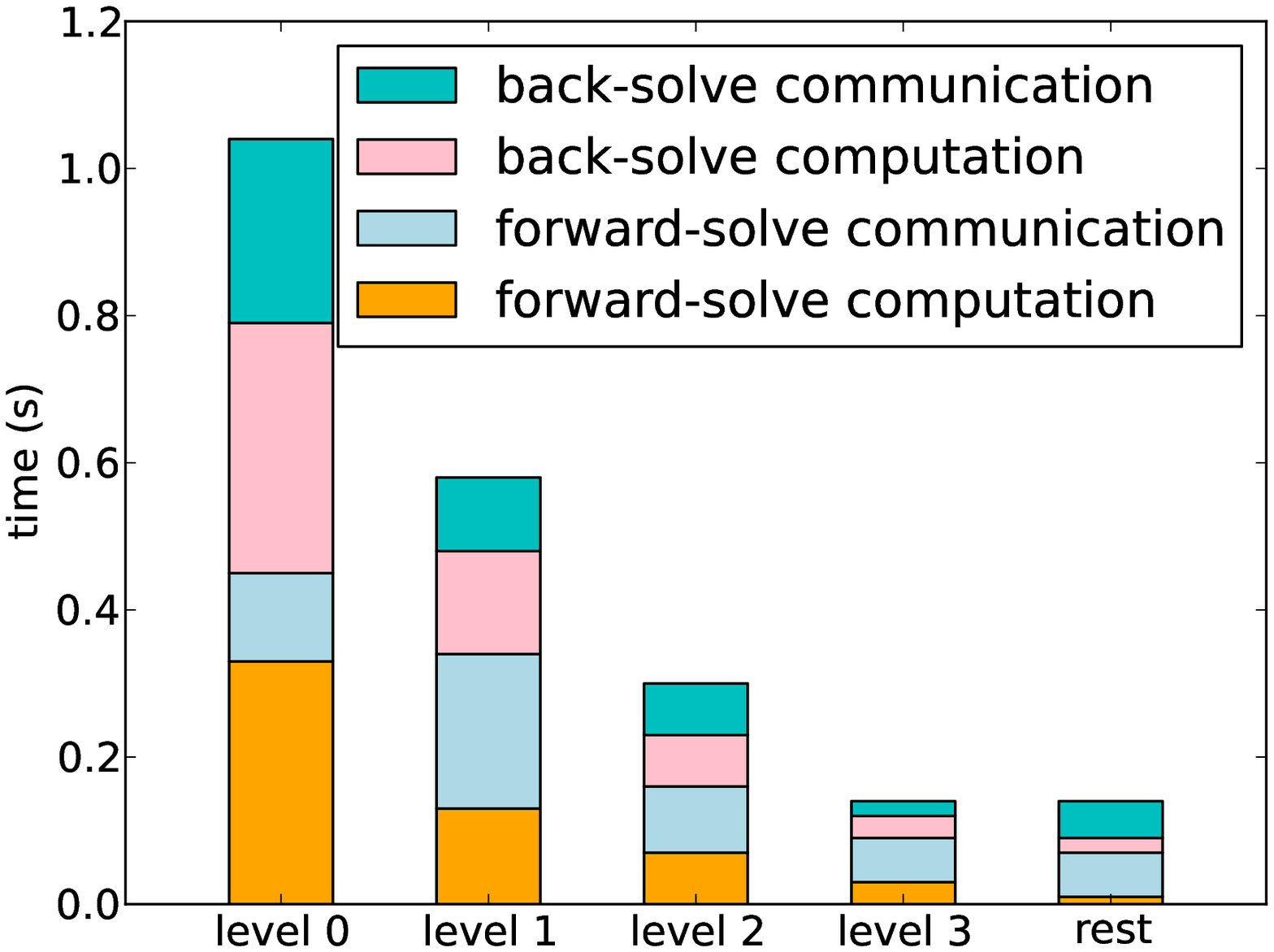}} 
\caption{11 vertical mesh layers. Breakdown of the factorization time and the solve time (forward substitution+backward substitution for one iteration) for the 4km resolution (on 64 processors). Note ``level 0" is the finest level in the hierarchical solver ($\epsilon=10^{-2}$).
\label{fig:level_time}}
\end{center}
\end{figure}

Fig.~\ref{fig:level_time} shows the breakdown of the factorization time and solve time (forward-substitution+backward-substitution) per iteration for different levels. As the figure shows, both the factorization time and the solve time decreases proportionally at coarser levels. The reason for this is that the number of partitions at the coarse level is halved while the size of every partition (twice the compression rank) remains bounded. This type of behavior is commonly observed in the profile of multi-level methods (e.g., the multigrid method and the fast multipole method) and is crucial for achieving parallel scalability.

\section{Conclusions and future work}

In this paper, we have introduced the deferred-compression technique for developing robust hierarchical solvers based on strongly admissible hierarchical matrices. For these matrices, off-diagonal matrix blocks that satisfy the strong admissibility condition are numerically low-rank (a.k.a., data-sparse). This low-rank property is leveraged in fast algorithms for computing approximate Cholesky factorizations of an SPD matrix, where (block) Gaussian elimination is applied after low-rank blocks are compressed. In the deferred-compression scheme, by contrast, these matrix blocks are first scaled by the inverse of the Cholesky factor of the corresponding diagonal block before low-rank approximations are applied. This deferred compression provably reduces the error in forming the subsequent Schur complement, especially for ill-conditioned linear systems. Our analysis shows that the $ww$ block in the Schur complement becomes second-order accurate ($\epsilon^2$) with respect to the truncation error $\epsilon$, as opposed to first-order accurate ($\epsilon/\sigma_{\min}(A_{ss})$) in the original algorithm, and more importantly, the block is shown to be SPD.

The effectiveness of the deferred-compression scheme is demonstrated through the newly developed improved LoRaSp solver, which is based on the original LoRaSp method and deploys the new compression technique. The improved LoRaSp solver has linear computational complexity under some mild assumptions, and its parallelization is similar to the original LoRaSp solver. Similar to ILU, the improved LoRaSp solver computes an approximate factorization by compressing fill-in blocks, but its dropping/truncation rule is based on the decay of singular values, which is expected to be more efficient than other level-based or threshold-based rules typically used in ILU. With a general graph partitioner, the improved LoRaSp solver can be used as a ``black-box'' method to solve general SPD sparse linear systems.

The application of ice sheet modeling is studied to benchmark the improved LoRaSp solver against other state-of-the-art methods. The standard smoothed aggregation AMG solver struggles due to difficulties associated with the strong anisotropic phenomena. On the other hand, ILU, a commonly used method in practical ice sheet simulations, has the disadvantage that the number of iterations doubles as the discretization mesh is refined, making it an $\mathcal{O}(N^{3/2})$ method. Compared with existing methods, our improved LoRaSp solver delivers a $\mathcal{O}(N)$ solution for a wide range of meshes. For extruded meshes used in ice sheet modeling, we have developed the extruded partitioning scheme to boost the performance of our solver, and we expect this approach to be effective for other geophysical modeling of thin structures.

Several directions for future research are as follows.
\begin{itemize}
\item The deferred-compression technique does not guarantee the subsequent Schur complement to be SPD. The creation of a numerical algorithm that guarantees the SPD property with strong admissibility is currently open.
\item The deferred-compression scheme and the improved LoRaSp solver were developed for SPD matrices; their extensions to non-symmetric matrices should be explored. For non-symmetric matrices, the optimal scaling factors for the upper triangular and the lower triangular parts need to be determined.
\item If the near-null space (very small singular values) of a physical model is available, it can be taken advantage of by hierarchical solvers to accelerate convergence. Such a scheme~\cite{CC:yang2016sparse} respects extra constraints on low-rank approximations and would resolve any vector in the near-null space exactly or very accurately.
\end{itemize}

\section{Acknowledments}
We thank Mauro Perego for help with the ice sheet test problems. This work was partly funded by the U.S.\ Department of Energy through the Predictive Science Academic Alliance Program (PSAAP II) under Award Number DE-NA0002373-1 and partly funded by an LDRD research grant from Sandia National Laboratories. Sandia National Laboratories is a multimission laboratory managed and operated by National Technology and Engineering Solutions of Sandia, LLC, a wholly owned subsidiary of Honeywell International, Inc., for the U.S.\ Department of Energy's National Nuclear Security Administration under contract DE-NA-0003525.





\bibliographystyle{model1-num-names}
\bibliography{ice_sheet.bib}

\begin{thebibliography}{35}
\expandafter\ifx\csname natexlab\endcsname\relax\def\natexlab#1{#1}\fi
\providecommand{\bibinfo}[2]{#2}
\ifx\xfnm\relax \def\xfnm[#1]{\unskip,\space#1}\fi
\bibitem[{Davis et~al.(2016)Davis, Rajamanickam, and
  Sid-Lakhdar}]{CC:davis2016survey}
\bibinfo{author}{T.~A. Davis}, \bibinfo{author}{S.~Rajamanickam},
  \bibinfo{author}{W.~M. Sid-Lakhdar},
\newblock \bibinfo{title}{A survey of direct methods for sparse linear
  systems},
\newblock \bibinfo{journal}{Acta Numerica} \bibinfo{volume}{25}
  (\bibinfo{year}{2016}) \bibinfo{pages}{383--566}.
\bibitem[{Hackbusch(1999)}]{CC:hackbusch1999sparse}
\bibinfo{author}{W.~Hackbusch},
\newblock \bibinfo{title}{A sparse matrix arithmetic based on
  $\mathcal{H}$-matrices. {Part} {I}: Introduction to $\mathcal{H}$-matrices},
\newblock \bibinfo{journal}{Computing} \bibinfo{volume}{62}
  (\bibinfo{year}{1999}) \bibinfo{pages}{89--108}.
\bibitem[{Hackbusch and Khoromskij(2000)}]{CC:hackbusch2000sparse}
\bibinfo{author}{W.~Hackbusch}, \bibinfo{author}{B.~N. Khoromskij},
\newblock \bibinfo{title}{A sparse $\mathcal{H}$-matrix arithmetic.},
\newblock \bibinfo{journal}{Computing} \bibinfo{volume}{64}
  (\bibinfo{year}{2000}) \bibinfo{pages}{21--47}.
\bibitem[{Hackbusch and B{\"o}rm(2002)}]{CC:hackbusch2002data}
\bibinfo{author}{W.~Hackbusch}, \bibinfo{author}{S.~B{\"o}rm},
\newblock \bibinfo{title}{Data-sparse approximation by adaptive
  $\mathcal{H}^2$-matrices},
\newblock \bibinfo{journal}{Computing} \bibinfo{volume}{69}
  (\bibinfo{year}{2002}) \bibinfo{pages}{1--35}.
\bibitem[{Hackbusch(2015)}]{CC:hackbusch2015mathcal}
\bibinfo{author}{W.~Hackbusch},
\newblock \bibinfo{title}{$\mathcal{H}^2$-matrices},
\newblock in: \bibinfo{booktitle}{Hierarchical Matrices: Algorithms and
  Analysis}, \bibinfo{publisher}{Springer}, \bibinfo{year}{2015}, pp.
  \bibinfo{pages}{203--240}.
\bibitem[{Xia et~al.(2010)Xia, Chandrasekaran, Gu, and Li}]{CC:xia2010fast}
\bibinfo{author}{J.~Xia}, \bibinfo{author}{S.~Chandrasekaran},
  \bibinfo{author}{M.~Gu}, \bibinfo{author}{X.~S. Li},
\newblock \bibinfo{title}{Fast algorithms for hierarchically semiseparable
  matrices},
\newblock \bibinfo{journal}{Numerical Linear Algebra with Applications}
  \bibinfo{volume}{17} (\bibinfo{year}{2010}) \bibinfo{pages}{953--976}.
\bibitem[{Chandrasekaran et~al.(2006)Chandrasekaran, Gu, and
  Pals}]{CC:chandrasekaran2006fast}
\bibinfo{author}{S.~Chandrasekaran}, \bibinfo{author}{M.~Gu},
  \bibinfo{author}{T.~Pals},
\newblock \bibinfo{title}{A fast {ULV} decomposition solver for hierarchically
  semiseparable representations},
\newblock \bibinfo{journal}{SIAM Journal on Matrix Analysis and Applications}
  \bibinfo{volume}{28} (\bibinfo{year}{2006}) \bibinfo{pages}{603--622}.
\bibitem[{Amestoy et~al.(2015)Amestoy, Ashcraft, Boiteau, Buttari, L'Excellent,
  and Weisbecker}]{CC:amestoy2015improving}
\bibinfo{author}{P.~Amestoy}, \bibinfo{author}{C.~Ashcraft},
  \bibinfo{author}{O.~Boiteau}, \bibinfo{author}{A.~Buttari},
  \bibinfo{author}{J.-Y. L'Excellent}, \bibinfo{author}{C.~Weisbecker},
\newblock \bibinfo{title}{Improving multifrontal methods by means of block
  low-rank representations},
\newblock \bibinfo{journal}{SIAM Journal on Scientific Computing}
  \bibinfo{volume}{37} (\bibinfo{year}{2015}) \bibinfo{pages}{A1451--A1474}.
\bibitem[{Aminfar et~al.(2016)Aminfar, Ambikasaran, and
  Darve}]{CC:aminfar2016fast}
\bibinfo{author}{A.~Aminfar}, \bibinfo{author}{S.~Ambikasaran},
  \bibinfo{author}{E.~Darve},
\newblock \bibinfo{title}{A fast block low-rank dense solver with applications
  to finite-element matrices},
\newblock \bibinfo{journal}{Journal of Computational Physics}
  \bibinfo{volume}{304} (\bibinfo{year}{2016}) \bibinfo{pages}{170--188}.
\bibitem[{Pouransari et~al.(2017)Pouransari, Coulier, and
  Darve}]{CC:pouransari2017fast}
\bibinfo{author}{H.~Pouransari}, \bibinfo{author}{P.~Coulier},
  \bibinfo{author}{E.~Darve},
\newblock \bibinfo{title}{Fast hierarchical solvers for sparse matrices using
  extended sparsification and low-rank approximation},
\newblock \bibinfo{journal}{SIAM Journal on Scientific Computing}
  \bibinfo{volume}{39} (\bibinfo{year}{2017}) \bibinfo{pages}{A797--A830}.
\bibitem[{Xia and Gu(2010)}]{CC:xia2010robust}
\bibinfo{author}{J.~Xia}, \bibinfo{author}{M.~Gu},
\newblock \bibinfo{title}{Robust approximate {Cholesky} factorization of
  rank-structured symmetric positive definite matrices},
\newblock \bibinfo{journal}{SIAM Journal on Matrix Analysis and Applications}
  \bibinfo{volume}{31} (\bibinfo{year}{2010}) \bibinfo{pages}{2899--2920}.
\bibitem[{Xia and Xin(2017)}]{CC:xia2017effective}
\bibinfo{author}{J.~Xia}, \bibinfo{author}{Z.~Xin},
\newblock \bibinfo{title}{Effective and robust preconditioning of general {SPD}
  matrices via structured incomplete factorization},
\newblock \bibinfo{journal}{SIAM Journal on Matrix Analysis and Applications}
  \bibinfo{volume}{38} (\bibinfo{year}{2017}) \bibinfo{pages}{1298--1322}.
\bibitem[{Xing and Chow(2018)}]{CC:xing2018preserving}
\bibinfo{author}{X.~Xing}, \bibinfo{author}{E.~Chow},
\newblock \bibinfo{title}{Preserving positive definiteness in hierarchically
  semiseparable matrix approximations},
\newblock \bibinfo{journal}{SIAM Journal on Matrix Analysis and Applications}
  \bibinfo{volume}{39} (\bibinfo{year}{2018}) \bibinfo{pages}{829--855}.
\bibitem[{Chen et~al.(2016)Chen, Rajamanickam, Boman, and
  Darve}]{CC:chen2016parallel}
\bibinfo{author}{C.~Chen}, \bibinfo{author}{S.~Rajamanickam},
  \bibinfo{author}{E.~G. Boman}, \bibinfo{author}{E.~Darve},
  \bibinfo{title}{Parallel hierarchical solver for elliptic partial
  differential equations}, \bibinfo{type}{Technical Report}, Sandia National
  Laboratories, \bibinfo{year}{2016}.
\bibitem[{Solomon(2007)}]{CC:solomon2007climate}
\bibinfo{author}{S.~Solomon}, \bibinfo{title}{Climate change 2007-the physical
  science basis: {Working} {Group} {I} contribution to the fourth assessment
  report of the IPCC}, volume~\bibinfo{volume}{4},
  \bibinfo{publisher}{Cambridge University Press}, \bibinfo{year}{2007}.
\bibitem[{Stocker(2014)}]{CC:stocker2014climate}
\bibinfo{author}{T.~Stocker}, \bibinfo{title}{Climate change 2013: the physical
  science basis: {Working} {Group} {I} contribution to the fifth assessment
  report of the {Intergovernmental} {Panel} on {Climate} {Change}},
  \bibinfo{publisher}{Cambridge University Press}, \bibinfo{year}{2014}.
\bibitem[{Tuminaro et~al.(2016)Tuminaro, Perego, Tezaur, Salinger, and
  Price}]{CC:tuminaro2016matrix}
\bibinfo{author}{R.~Tuminaro}, \bibinfo{author}{M.~Perego},
  \bibinfo{author}{I.~Tezaur}, \bibinfo{author}{A.~Salinger},
  \bibinfo{author}{S.~Price},
\newblock \bibinfo{title}{A matrix dependent/algebraic multigrid approach for
  extruded meshes with applications to ice sheet modeling},
\newblock \bibinfo{journal}{SIAM Journal on Scientific Computing}
  \bibinfo{volume}{38} (\bibinfo{year}{2016}) \bibinfo{pages}{C504--C532}.
\bibitem[{Van{\v{e}}k et~al.(1996)Van{\v{e}}k, Mandel, and
  Brezina}]{CC:vanvek1996algebraic}
\bibinfo{author}{P.~Van{\v{e}}k}, \bibinfo{author}{J.~Mandel},
  \bibinfo{author}{M.~Brezina},
\newblock \bibinfo{title}{Algebraic multigrid by smoothed aggregation for
  second and fourth order elliptic problems},
\newblock \bibinfo{journal}{Computing} \bibinfo{volume}{56}
  (\bibinfo{year}{1996}) \bibinfo{pages}{179--196}.
\bibitem[{Karypis and Kumar(1998)}]{CC:karypis1998fast}
\bibinfo{author}{G.~Karypis}, \bibinfo{author}{V.~Kumar},
\newblock \bibinfo{title}{A fast and high quality multilevel scheme for
  partitioning irregular graphs},
\newblock \bibinfo{journal}{SIAM Journal on scientific Computing}
  \bibinfo{volume}{20} (\bibinfo{year}{1998}) \bibinfo{pages}{359--392}.
\bibitem[{Chevalier and Pellegrini(2008)}]{CC:chevalier2008pt}
\bibinfo{author}{C.~Chevalier}, \bibinfo{author}{F.~Pellegrini},
\newblock \bibinfo{title}{{PT-Scotch}: A tool for efficient parallel graph
  ordering},
\newblock \bibinfo{journal}{Parallel computing} \bibinfo{volume}{34}
  (\bibinfo{year}{2008}) \bibinfo{pages}{318--331}.
\bibitem[{Boman et~al.(2012)Boman, {\c{C}}ataly{\"u}rek, Chevalier, and
  Devine}]{CC:boman2012zoltan}
\bibinfo{author}{E.~G. Boman}, \bibinfo{author}{{\"U}.~V.
  {\c{C}}ataly{\"u}rek}, \bibinfo{author}{C.~Chevalier}, \bibinfo{author}{K.~D.
  Devine},
\newblock \bibinfo{title}{The {Zoltan} and {Isorropia} parallel toolkits for
  combinatorial scientific computing: Partitioning, ordering and coloring},
\newblock \bibinfo{journal}{Scientific Programming} \bibinfo{volume}{20}
  (\bibinfo{year}{2012}) \bibinfo{pages}{129--150}.
\bibitem[{Ho and Ying(2016)}]{ho2016hierarchical}
\bibinfo{author}{K.~L. Ho}, \bibinfo{author}{L.~Ying},
\newblock \bibinfo{title}{Hierarchical interpolative factorization for elliptic
  operators: differential equations},
\newblock \bibinfo{journal}{Communications on Pure and Applied Mathematics}
  \bibinfo{volume}{69} (\bibinfo{year}{2016}) \bibinfo{pages}{1415--1451}.
\bibitem[{Chen et~al.(2018)Chen, Tuminaro, Rajamanickam, Boman, and
  Darve}]{CC:chen2018hierarchical}
\bibinfo{author}{C.~Chen}, \bibinfo{author}{R.~Tuminaro},
  \bibinfo{author}{S.~Rajamanickam}, \bibinfo{author}{E.~G. Boman},
  \bibinfo{author}{E.~Darve},
\newblock \bibinfo{title}{A hierarchical solver for extruded meshes with
  applications to ice sheet modeling},
\newblock in: \bibinfo{booktitle}{Center for Computing Research Summer
  Proceedings 2017, A.D. Baczewski and M.L. Parks, eds., Technical Report
  SAND2018-2780O}, \bibinfo{organization}{Sandia National Laboratories}, pp.
  \bibinfo{pages}{3--18}.
\bibitem[{Bebendorf and Hackbusch(2003)}]{bebendorf2003existence}
\bibinfo{author}{M.~Bebendorf}, \bibinfo{author}{W.~Hackbusch},
\newblock \bibinfo{title}{Existence of $\cal h$-matrix approximants to the
  inverse fe-matrix of elliptic operators with $l^\infty$-coefficients},
\newblock \bibinfo{journal}{Numerische Mathematik} \bibinfo{volume}{95}
  (\bibinfo{year}{2003}) \bibinfo{pages}{1--28}.
\bibitem[{Bebendorf(2005)}]{bebendorf2005efficient}
\bibinfo{author}{M.~Bebendorf},
\newblock \bibinfo{title}{Efficient inversion of the galerkin matrix of general
  second-order elliptic operators with nonsmooth coefficients},
\newblock \bibinfo{journal}{Mathematics of Computation} \bibinfo{volume}{74}
  (\bibinfo{year}{2005}) \bibinfo{pages}{1179--1199}.
\bibitem[{Chandrasekaran et~al.(2010)Chandrasekaran, Dewilde, Gu, and
  Somasunderam}]{chandrasekaran2010numerical}
\bibinfo{author}{S.~Chandrasekaran}, \bibinfo{author}{P.~Dewilde},
  \bibinfo{author}{M.~Gu}, \bibinfo{author}{N.~Somasunderam},
\newblock \bibinfo{title}{On the numerical rank of the off-diagonal blocks of
  schur complements of discretized elliptic pdes},
\newblock \bibinfo{journal}{SIAM Journal on Matrix Analysis and Applications}
  \bibinfo{volume}{31} (\bibinfo{year}{2010}) \bibinfo{pages}{2261--2290}.
\bibitem[{Greengard and Rokhlin(1987)}]{greengard1987fast}
\bibinfo{author}{L.~Greengard}, \bibinfo{author}{V.~Rokhlin},
\newblock \bibinfo{title}{A fast algorithm for particle simulations},
\newblock \bibinfo{journal}{Journal of computational physics}
  \bibinfo{volume}{73} (\bibinfo{year}{1987}) \bibinfo{pages}{325--348}.
\bibitem[{Greengard and Rokhlin(1997)}]{greengard1997new}
\bibinfo{author}{L.~Greengard}, \bibinfo{author}{V.~Rokhlin},
\newblock \bibinfo{title}{A new version of the fast multipole method for the
  laplace equation in three dimensions},
\newblock \bibinfo{journal}{Acta numerica} \bibinfo{volume}{6}
  (\bibinfo{year}{1997}) \bibinfo{pages}{229--269}.
\bibitem[{Tezaur et~al.(2015{\natexlab{a}})Tezaur, Perego, Salinger, Tuminaro,
  and Price}]{CC:tezaur2015albany}
\bibinfo{author}{I.~K. Tezaur}, \bibinfo{author}{M.~Perego},
  \bibinfo{author}{A.~G. Salinger}, \bibinfo{author}{R.~S. Tuminaro},
  \bibinfo{author}{S.~F. Price},
\newblock \bibinfo{title}{Albany/{FELIX}: a parallel, scalable and robust,
  finite element, first-order stokes approximation ice sheet solver built for
  advanced analysis},
\newblock \bibinfo{journal}{Geoscientific Model Development}
  \bibinfo{volume}{8} (\bibinfo{year}{2015}{\natexlab{a}})
  \bibinfo{pages}{1197--1220}.
\bibitem[{Tezaur et~al.(2015{\natexlab{b}})Tezaur, Tuminaro, Perego, Salinger,
  and Price}]{CC:tezaur2015scalability}
\bibinfo{author}{I.~K. Tezaur}, \bibinfo{author}{R.~S. Tuminaro},
  \bibinfo{author}{M.~Perego}, \bibinfo{author}{A.~G. Salinger},
  \bibinfo{author}{S.~F. Price},
\newblock \bibinfo{title}{On the scalability of the {Albany}/{FELIX}
  first-order stokes approximation ice sheet solver for large-scale simulations
  of the {Greenland} and {Antarctic} ice sheets},
\newblock \bibinfo{journal}{Procedia Computer Science} \bibinfo{volume}{51}
  (\bibinfo{year}{2015}{\natexlab{b}}) \bibinfo{pages}{2026--2035}.
\bibitem[{Cuffey and Paterson(2010)}]{CC:cuffey2010physics}
\bibinfo{author}{K.~M. Cuffey}, \bibinfo{author}{W.~S.~B. Paterson},
  \bibinfo{title}{The physics of glaciers}, \bibinfo{publisher}{Academic
  Press}, \bibinfo{year}{2010}.
\bibitem[{Nye(1957)}]{CC:nye1957distribution}
\bibinfo{author}{J.~Nye},
\newblock \bibinfo{title}{The distribution of stress and velocity in glaciers
  and ice-sheets},
\newblock in: \bibinfo{booktitle}{Proceedings of the Royal Society of London A:
  Mathematical, Physical and Engineering Sciences}, volume
  \bibinfo{volume}{239}, \bibinfo{organization}{The Royal Society}, pp.
  \bibinfo{pages}{113--133}.
\bibitem[{MacAyeal et~al.(1996)MacAyeal, Rommelaere, Huybrechts, Hulbe,
  Determann, and Ritz}]{CC:macayeal1996ice}
\bibinfo{author}{D.~R. MacAyeal}, \bibinfo{author}{V.~Rommelaere},
  \bibinfo{author}{P.~Huybrechts}, \bibinfo{author}{C.~L. Hulbe},
  \bibinfo{author}{J.~Determann}, \bibinfo{author}{C.~Ritz},
\newblock \bibinfo{title}{An ice-shelf model test based on the {Ross} {Ice}
  {Shelf}, {Antarctica}},
\newblock \bibinfo{journal}{Annals of Glaciology} \bibinfo{volume}{23}
  (\bibinfo{year}{1996}) \bibinfo{pages}{46--51}.
\bibitem[{Devine et~al.(2006)Devine, Boman, Heaphy, Bisseling, and
  Catalyurek}]{devine2006parallel}
\bibinfo{author}{K.~D. Devine}, \bibinfo{author}{E.~G. Boman},
  \bibinfo{author}{R.~T. Heaphy}, \bibinfo{author}{R.~H. Bisseling},
  \bibinfo{author}{U.~V. Catalyurek},
\newblock \bibinfo{title}{Parallel hypergraph partitioning for scientific
  computing},
\newblock in: \bibinfo{booktitle}{Parallel and Distributed Processing
  Symposium, 2006. IPDPS 2006. 20th International},
  \bibinfo{organization}{IEEE}, pp. \bibinfo{pages}{10--pp}.
\bibitem[{Yang et~al.(2016)Yang, Pouransari, and Darve}]{CC:yang2016sparse}
\bibinfo{author}{K.~Yang}, \bibinfo{author}{H.~Pouransari},
  \bibinfo{author}{E.~Darve},
\newblock \bibinfo{title}{Sparse hierarchical solvers with guaranteed
  convergence},
\newblock \bibinfo{journal}{arXiv preprint arXiv:1611.03189}
  (\bibinfo{year}{2016}).

\end{thebibliography}







\end{document}